%% file: ms.tex
\pdfoutput=1

\documentclass{article}

\usepackage{arxiv}

\newtheorem{theorem}{Theorem}

\newtheorem{lemma}{Lemma}
\newtheorem{corollary}{Corollary}

\newcommand{\norm}[1]{\left\lVert#1\right\rVert}
\theoremstyle{definition}

\DeclareMathOperator{\Exp}{\mathrm{\mathbb{E}}}

\newcommand{\sgd}{\textsc{Sgd}\xspace} 
\newcommand{\igd}{\textsc{Igd}\xspace}
\newcommand{\gd}{\textsc{Gd}\xspace}
\newcommand{\rsgd}{\textsc{RandomShuffle}\xspace}
\newcommand{\reals}{\mathbb{R}}
\newcommand{\Oc}{\mathcal{O}}
\newcommand{\nf}{\nicefrac}

\numberwithin{equation}{section}

\begin{document}

\title{Random Shuffling Beats SGD after Finite Epochs}

\author{% -- this % sign is essential; don't remove
  \name Jeff Z. HaoChen \email{jhaochen@stanford.edu}\\
  \addr{Stanford University}\\[4pt]
  \name Suvrit Sra \email{suvrit@mit.edu}\\
  \addr{Massachusetts Institute of Technology}
}

\maketitle
 
\input{introduction}

\input{relatedwork}
\input{main}
\input{sparse}
\input{sketch}
\input{discus}
\input{extension}

%\small
\bibliographystyle{abbrvnat}
\setlength{\bibsep}{3pt}
\bibliography{refer}

\newpage

\normalsize
\appendix
\input{apx1}

\input{apx2}

\end{document}

%% file: introduction.tex
\begin{abstract}
A long-standing problem in the theory of stochastic gradient descent (\sgd) is to prove that its without-replacement version \rsgd converges faster than the usual with-replacement version. We present the \emph{first} (to our knowledge) non-asymptotic solution to this problem, which shows that after a ``reasonable'' number of epochs \rsgd indeed converges faster than \sgd. Specifically, we prove that under strong convexity and second-order smoothness, the sequence generated by \rsgd converges to the optimal solution at the rate $\mathcal{O}\bigl(\frac{1}{T^2} + \frac{n^3}{T^3}\bigr)$, where $n$ is the number of components in the objective, and $T$ is the total number of iterations. This result shows that after a reasonable number of epochs \rsgd is \emph{strictly better} than \sgd (which converges as $\mathcal{O}\bigl(\frac{1}{T}\bigr)$). The key step toward showing this better dependence on $T$ is the introduction of $n$ into the bound; and as our analysis will show, in general a dependence on $n$ is unavoidable without further changes to the algorithm. We show that for sparse data \rsgd has the rate $\mathcal{O}\left(\frac{1}{T^2}\right)$, again strictly better than \sgd. Furthermore, we discuss extensions to nonconvex gradient dominated functions, as well as non-strongly convex settings.
\end{abstract}

\section{Introduction}
We consider stochastic optimization methods for the finite-sum problem
\begin{align}
  F\left(x\right) := \frac{1}{n} \sum_{i=1}^n f_i\left(x\right),\label{1}
\end{align}
where each function $f_i : \reals^d\rightarrow \reals$ is smooth and convex, and the sum $F$ is strongly convex. A classical approach to solving~\eqref{1} is stochastic gradient descent (\sgd). At each iteration \sgd independently samples an index $i$ uniformly from $\{1,\ldots,n\}$, and uses the (stochastic) gradient $\nabla f_i$ to compute its update. The stochasticity makes each iteration of \sgd cheap, and the uniformly independent sampling of $i$ makes $\nabla f_i$ an unbiased estimator of the full gradient $\nabla F$. These properties are central to \sgd's effectiveness in large scale machine learning, and underlie much of its theoretical analysis (see for instance,  \citep{sra2012optimization,rakhlin2012making,bertsekas2011incremental,bottou2016optimization,shalev2014understanding}).

However, what is actually used in practice is the \emph{without replacement} version of \sgd, henceforth called \rsgd. Specifically, at each epoch \rsgd samples a random permutation of the $n$ functions uniformly independently (some implementations shuffle the data only once at load, rather than at each epoch). Then, it iterates over these functions according to the sampled permutation and updates in a manner similar to \sgd. Avoiding the use of random sampling at each iteration, \rsgd can be computationally more practical~\citep{bottou2012stochastic}; furthermore, as one would expect, empirically \rsgd is known to converge faster than \sgd~\citep{bottou2009curiously}.

This discrepancy between theory and practice has been a long-standing problem in the theory of \sgd. It has drawn renewed attention recently, with the goal of better understanding convergence of \rsgd. The key difficulty is that without-replacement leads to statistically non-independent samples, which greatly complicates analysis.  Two extreme case positive results are however available:  \citet{shamir2016without} shows that \rsgd is not much worse than usual \sgd, provided the number of epochs is not too large; while \citet{gurbuzbalaban2015random} show that \rsgd converges faster than \sgd asymptotically at the rate $\mathcal{O}(\frac{1}{T^2})$.

But it remains unclear what happens in between, after a reasonable finite number of epochs are run. This regime is the most compelling one to study, since in practice one runs neither one nor infinitely many epochs. This motivates the central question of our paper:
\begin{center}
  \it Does \rsgd converge faster than \sgd after a reasonable number of epochs?
\end{center}
We answer this question positively in this paper; our results are more precisely summarized below.

\begin{table}
  \caption{\small Comparison of convergence rates of \sgd and \rsgd. The first three functions considered are strongly convex. We omit all the constants from the rate (for details on  constants, please see Section~\ref{discuss}). Under the sparse setting (sparsity level $\rho$), we are not aware of specialized results corresponding to \sgd. For the LP condition functions, see definition in section~\ref{lpcondition}. Here the criterion of LP condition functions is based on suboptimal of function values, while the other three are based on distance to the unique optimal solution.}
  \label{sample-table}
  \centering
  \begin{tabular}{lllll}
    \toprule
  Algorithm & Quadratic & Lipschitz Hessian & Sparse Data & LP Condition \\
    \midrule
  \sgd & $\Oc(1/T)$& $\Oc(1/T)$ & $\Oc(1/T)$ & $\Oc(1/T)$\\
  \addlinespace[0.2em]
  %Gradient Descent & $e^{-T/n}$ & $e^{-T/n}$ & $-$ \\
  \rsgd & \textcolor{red}{$\Oc(1/T^2 + n^3/T^3)$} &  \textcolor{red}{$\Oc(1/T^2 + n^3/T^3)$}& \textcolor{red}{$\Oc(1/T^2 + \rho^2 n^3 / T^3)$} & \textcolor{red}{$\Oc(1/T^2 + n^3/T^3)$}\\
    \bottomrule
  \end{tabular}
\end{table}

%\begin{table}[t]\small
%\newcolumntype{R}{>{\centering\arraybackslash}X}
%\setlength{\extrarowheight}{2pt}
%\captionsetup{ singlelinecheck=false}
%\begin{tabularx}{\textwidth}{ |R|R|R|R| }
%  \hline
%  Algorithm & Quadratic & Lipschitz Hessian & Sparse Data \\
%  \hline 
%  \sgd & $\Oc(1/T)$& $\Oc(1/T)$ & $\Oc(1/T)$ \\
  %Gradient Descent & $e^{-T/n}$ & $e^{-T/n}$ & $-$ \\
%  \rsgd & \textcolor{red}{$\Oc(1/T^2 + n^3/T^3)$} &  \textcolor{red}{$\Oc(1/T^2 + n^3/T^3)$}& \textcolor{red}{$\Oc(1/T^2 + \rho^2 n^3 / T^2)$}\\
%  \hline
%\end{tabularx}
%\caption{\small Comparison of convergence rates of \sgd and \rsgd. All functions  considered are strongly convex. We omit all the constants from the rate (for details on  constants, please see Section~\ref{discuss}). Under the sparse setting (sparsity level $\rho$), we are not aware of specialized results corresponding to \sgd.}
%\end{table}

\subsection{Summary of results}
We follow the common practice of reporting convergence rates depending on $T$, the number of calls to the (stochastic / incremental) gradient oracle. % For problem~\eqref{1}, this number corresponds to the number of gradient at some value for one component function. 
For instance, \sgd converges at the rate $\mathcal{O}(\frac{1}{T})$ for  solving~\eqref{1}, ignoring logarithmic terms in the bound~\citep{rakhlin2012making}. The underlying argument is to view \sgd as stochastic approximation with noise~\citep{nemirov09}, therefore ignoring the finite-sum structure of~\eqref{1}. Our key observation for \rsgd is that one should reasonably include dependence on $n$ into the bound (see Section~\ref{dpn}). Such a compromise leads to a better dependence on $T$, which further shows how \rsgd beats \sgd after a finite number of epochs. Our main contributions are the following:

\begin{list}{{\small$\blacktriangleright$}}{\leftmargin=2em}
\item Under a mild assumption on second order differentiability, and assuming strong-convexity, we establish a convergence rate of $\mathcal{O}\bigl(\frac{1}{T^2} + \frac{n^3}{T^3}\bigr)$ for \rsgd, where $n$ is the number of components in~\eqref{1}, and $T$ is the total number of iterations (Theorem~\ref{thm41} and \ref{thm43}). From the bounds we can calculate the precise number of epochs after which \rsgd is \emph{strictly better} than \sgd. 
\item We prove that a dependence on $n$ is necessary for beating the \sgd rate $\Oc(\nf{1}{T})$. This tradeoff precludes the possibility of proving a convergence rate of the type $\mathcal{O}\bigl(\nf{1}{T^{1+\delta}}\bigr)$ with some $\delta>0$ in the general case, and justifies our choice of introducing $n$ into the rate (Theorem~\ref{thm42}).
\item Assuming a sparse data setting common in machine learning, we further improve the convergence rate of \rsgd to $\mathcal{O}\bigl(\nf{1}{T^2}\bigr)$. This rate is \emph{strictly better} than \sgd, indicating \rsgd's advantage in such cases (Theorem~\ref{thm52}).
\item We extend our results to the non-convex function class with Polyak-\L ojasiewicz condition, establishing a similar $\mathcal{O}\bigl(\frac{1}{T^2} + \frac{n^3}{T^3}\bigr)$ rate for \rsgd (Theorem~\ref{thpl}).
\item We show a class of examples where \rsgd is provably faster than \sgd after arbitrary number (even less than one epoch) of iterations (Theorem~\ref{last}).
\end{list}
%We note that unlike the asymptotic analysis in~\citep{gurbuzbalaban2015random}, we focus on \rsgd's convergence after a reasonable number of finite epochs. In particular, the convergence rate that we obtain has a better dependence on $T$ than the $\Oc\bigl(\frac{1}{T}\bigr)$ rate of \sgd after just $\Oc(\sqrt{n})$ epochs are run.
We provide a detailed discussion of various aspects of our results in Section~\ref{discuss}, including explicit comparisons to \sgd, the role of condition numbers, as well as some limitations. Finally, we end by noting some extensions and open problems in Section~\ref{sec:conc}. As one of the extensions, for non-strongly convex problems, we prove that \rsgd achieves a comparable convergence rate as \sgd, with possibly smaller constant in the bound under certain parameter paradigms (Theorem~\ref{ext1}). 

%in terms of bounds.
\subsection{Related work}
\citet{Recht2012BeneathTV} conjecture a tantalizing matrix AM-GM inequality that underlies \rsgd's superiority over \sgd. While limited progress on this conjecture has  been reported~\citep{israel2016arithmetic,zhang2014note}, the correctness of the full conjecture is still wide open. With the technique of transductive Rademacher complexity, \citet{shamir2016without} shows that \sgd is not worse than \rsgd provided the number of iterations is not too large. Asymptotic analysis is provided in~\citep{gurbuzbalaban2015random}, which proves that \rsgd limits to a $\mathcal{O}\bigl(\nf{1}{T^2}\bigr)$ rate for large $T$. \citet{ying2018stochastic} show that for a fixed step size, \rsgd converges to a distribution closer to optimal than \sgd asymptotically. 

When the functions are visited in a deterministic order (e.g., cyclic), the method turns into Incremental Gradient Descent (\igd), which has a long history~\citep{bertsekas2011incremental}. \citet{kohonen1974adaptive} shows that \igd converges to a limit cycle under constant step size and quadratic functions.  Convergence to neighborhood of optimality for more general functions is studied in several works, under the assumption that step size is bounded away from zero (see for instance~\citep{solodov1998incremental}). With properly diminishing step size, \citet{nedic2001convergence} show that an $\mathcal{O}\bigl(\nf{1}{\sqrt{T}}\bigr)$ convergence rate in terms of distance to optimal can be achieved under strong convexity of the finite-sum. This rate is further improved in \citep{gurbuzbalaban2015convergence} to $\mathcal{O}(\nf{1}{T})$ under a second order differentiability assumption.

In the real world, \rsgd has been proposed as a standard heuristic~\citep{bottou2012stochastic}. With numerical experiments, \citet{bottou2009curiously} notices an approximately $\mathcal{O}(\nf{1}{T^2})$ convergence rate of \rsgd. Without-replacement sampling also improves data-access efficiency in distributed settings, see for instance~\citep{feng2012towards, lee2015distributed}. The permutation-sampling idea has been further embedded into more complicated algorithms; see~\citep{de2016efficient,defazio2014finito,shamir2016without} for variance-reduced methods, and \citep{shalev2013stochastic} for decomposition methods. %or \sgd algorithms for more specific optimization tasks~\citep{recht2013parallel}. 

Finally, we note a related body of work on coordinate descent, where a similar problem has been studied: \emph{when does random permutation over coordinates behave well?} \citet{Grbzbalaban2017WhenCC} give two kinds of quadratic problems when cyclic version of coordinate descent beats the with replacement one, which is a stronger result indicating that random permutation also beats the with replacement method. However, such a deterministic version of the algorithm suffers from poor worst case. Indeed, in \citep{sun2016worst} a setting is analyzed where cyclic coordinate descent can be dramatically worse than both with-replacement and random permutation versions of coordinate descent. \citet{lee2016random} further study this setting, and analyze how the random permutation version of coordinate descent avoids the slow convergence of cyclic version. In~\citep{wright2017analyzing}, Wright et el.\ propose a more general class of quadratic functions where random permutation outperforms cyclic coordinate descent. %However, all of these results rely on the exact minimization for each iteration, which is quite different from what happens in finite sum setting. 

%%% Local Variables:
%%% mode: latex
%%% TeX-master: "paper"
%%% End:

%% file: relatedwork.tex
\section{Background and problem setup}
%\subsection{Settings and Notations}
For problem~\eqref{1}, we assume the finite sum function $F(x): \mathbb{R}^d\rightarrow \mathbb{R}$ is strongly convex, i.e., \[F(x) \geq F(y) + \left\langle \nabla F(y), x-y \right\rangle + \tfrac{\mu}{2}\norm{x-y}^2,\] where $x, y\in \mathbb{R}^d$, and $\mu>0$ is the strong convexity parameter. Furthermore, we assume each component function is $L$-smooth, so that for $i=1,\ldots,n$, there exists a constant $L$ such that
\begin{equation}
  \label{eq:4}
  \norm{\nabla f_i(x) - \nabla f_i(y)} \leq L \norm{x-y}.
\end{equation}
%for $i = 1,\cdots, n$, where $L$ is the Lipschitz constant. 

Furthermore, we assume that the component functions are second order differentiable with a Lipschitz continuous Hessian. We use $H_i(x)$ to denote the Hessian of function $f_i$ at $x$. Specifically, for each $i=1,\ldots,n$, we assume that for all $x, y \in \reals^d$, there exists a constant $L_H$ such that
\begin{equation}
  \label{eq:5}
  \norm{H_i(x) - H_i(y)} \leq L_H \norm{x-y}.
\end{equation}
The norm is the spectral norm for matrices and $\ell_2$ norm for vectors.
We denote the unique minimizer of $F(x)$ as $x^*$, the index set $\{1, \cdots, n\}$ as $[n]$. The complexity bound is represented as $\mathcal{O}(\cdot)$, with all logarithmic terms hidden. All other parameters that might be hidden in the complexity bounds will be clarified in corresponding sections.

\subsection{The algorithms under study: \sgd and \rsgd}
For both \sgd and \rsgd, we use $\gamma$ as the step size, which is predetermined before the algorithms are run. The sequences generated by both methods are denoted as $(x_k)_{k=0}^T$; here $x_0$ is the initial point and $T$ is the total number of iterations (i.e., number of stochastic gradients used).

\sgd is defined as follows: for each iteration $1\leq k\leq T$, it picks an index $s(k)$ independently uniformly from the index set $[n]$, and then performs the update
\begin{equation}
\label{eq:1}\tag{\sgd}
x_{k} = x_{k-1} - \gamma \nabla f_{s(k)}(x_{k-1}).
\end{equation}
In contrast, \rsgd runs as follows: for each epoch $t$, it picks one permutation $\sigma_t(\cdot): [n]\rightarrow [n]$ independently uniformly from the set of all permutations of $[n]$. Then, it sequentially visits each of the component functions of the finite-sum~\eqref{1} and performs the update
\begin{equation}
  \label{eq:2}\tag{\rsgd}
  \hspace{1.9in}x^t_{k} = x^t_{k-1} - \gamma \nabla f_{\sigma_t(k)}\left(x^t_{k-1}\right),
\end{equation}
for $1\le k \le n$. Here $x^t_k = x_{(t-1)n+k}$ represents the $k$-th iterate within the $t$-th epoch. For two consecutive epochs $t$ and $t+1$, one has $x^{t+1}_0 = x^t_{n}$; for the initial point one has $x^1_0 = x_0$. For convenience of analysis, we always assume \rsgd is run for an integer number of epochs, i.e., $T=ln$ for some $l\in \mathbb{Z}^+$. This is a reasonable assumption given our main interest is when several  epochs of \rsgd are run. %According to~\citep{shamir2016without}, one can think the last extra iterations \note{[CLARIFY?]} as no worse than \sgd.

%%% Local Variables:
%%% mode: latex
%%% TeX-master: "paper"
%%% End:

%% file: main.tex
\section{Convergence analysis of \rsgd}
\label{sec:convg}
The goal of this section is to build theoretical analysis for \rsgd. Specifically, we answer the following question: \emph{when can we show \rsgd to be better than \sgd}? We begin by first analyzing  % The structure of this section is as following: in subsection~\ref{qc}, we introduce our result on the standing example of
quadratic functions in Section~\ref{qc}, where the analysis benefits from having a  constant Hessian. Subsequently, in Section~\ref{gc}, we extend our analysis to the general (smooth) strongly convex setting. A key idea in our analysis is to make the convergence rate bounds sensitive to $n$, the number of components in the finite-sum~\eqref{1}. In Section~\ref{dpn}, we discuss and justify the necessity of introducing $n$ into our convergence bound.

\subsection{\rsgd for quadratics}\label{qc}
We first consider the quadratic instance of~\eqref{1}, where 
\begin{equation}
  \label{eq:3}
  f_i(x) = \tfrac{1}{2}x^T A_i x + b_i^Tx,\qquad i=1,\ldots,n,
\end{equation}
where $A_i\in \reals^{d\times d}$ is positive semi-definite, and $b_i \in \reals^d$. We should notice often in analyzing strongly convex problems, the quadratic case presents a good example when tight bounds are achieved. 

Quadratic functions have a constant Hessian function $H_i(x) = A_i$, which eases our analysis. Similar to the usual \sgd, our bound also depends on the following constants: % \begin{list}{$\bullet$}{\leftmargin=1.5em}\it
(i) strong convexity parameter $\mu$, and component-wise Lipschitz constant $L$; (ii) diameter bound $\norm{x - x^*} \leq D$ (i.e., any iterate $x$ remains bounded; can be enforced by explicit projection if needed); and (iii) bounded gradients $\norm{\nabla f_i(x)} \leq G$ for each $f_i$ ($1\le i \le n$), and any $x$ satisfying (ii). %for which $\norm{x - x^*} \leq D$.}
%\end{list}
We omit these constants for clarity, but discuss the condition number further in Section~\ref{discuss}.

Our main result for \rsgd is the following (omitting logarithmic terms): 
\begin{theorem}\label{thm41}
  With $f_i$ defined by~\eqref{eq:3}, let the condition number of problem \eqref{1} be $\kappa = L/\mu$. So long as $\frac{T}{\log T} > 6(1+\kappa)n$,  with step size $\gamma = \frac{4\log T}{ T\mu}$,  \rsgd achieves convergence rate: \[\Exp [\norm{x_T - x^*}^2] \leq \mathcal{O}\Bigl(\frac{1}{T^2} + \frac{n^3}{T^3}\Bigr).\]
\end{theorem}
We provide a proof sketch in Section~\ref{prf}, deferring the fairly involved technical details to Appendix~\ref{th1}. In terms of sample complexity, Theorem~\ref{thm41} yields the following corollary:
\begin{corollary}
  Let $f_i$ be defined by~\eqref{eq:3}. The sample complexity for \rsgd to achieve $\Exp[\norm{x_T - x^*}^2]=\Oc(\epsilon)$ is no more than $\mathcal{O}\bigl( \epsilon^{-\nf12} + n \epsilon^{-\nicefrac{1}{3}} \bigr)$.
\end{corollary}

We observe that in the regime when $T$ gets large, our result matches~\citep{gurbuzbalaban2015random}. But it provides more information when the number of epochs is not so large that the $\frac{n^3}{T^3}$ can be neglected. This setting is clearly the most compelling to study. Formally, we recover the main result of~\citep{gurbuzbalaban2015random} as the following:
\begin{corollary}
  As $T\rightarrow \infty$, \rsgd achieves asymptotic convergence rate $\mathcal{O}\left(\frac{1}{T^2}\right)$ when run with the proper step size schedule.
\end{corollary}

\subsection{\rsgd for strongly convex problems}\label{gc}
Next, we consider the more general case where each component function $f_i$ is convex and the sum $F(x)=\tfrac{1}{n}\sum_if_i(x)$ is strongly convex. Surprisingly\footnote{Intuitively, the change of Hessian over the domain can raise challenges. However, our convergence rate here is quite similar to quadratic case, with only mild dependence on Hessian Lipschitz constant. }, one can easily adapt the methodology of the proof for Theorem~\ref{thm41} in this setting. To this end, our analysis requires one further assumption that each component function is second order differentiable and its Hessian satisfies the Lipschitz condition~\eqref{eq:5} with constant $L_H$. % \[\norm{H_i(x) - H_i(y)} \leq L_H \norm{x-y}.\] Function $H_i(x)$ is the Hessian matrix of $f_i$ at point $x$, $L_H$ is some positive constant.

Under these assumptions, we obtain the following result:
\begin{theorem}\label{thm43}
  Define constant $C = \max\left\{\frac{32}{\mu^2}(L_H LD + 3L_H G), 12(1+\frac{L}{\mu}) \right\}$. So long as $\frac{T}{\log T} > Cn$,  with step size $\eta = \frac{8\log T}{ T\mu}$, \rsgd achieves convergence rate: \[\Exp [\norm{x_T - x^*}^2] \leq \mathcal{O}\Bigl(\frac{1}{T^2} + \frac{n^3}{T^3}\Bigr).\]
\end{theorem}
Except for extra dependence on $L_H$ and a mildly different step size, this rate is essentially the same as that in quadratic case. The proof for the result can be found in Appendix~\ref{generalproof}. Due to the similar formulation, most of the consequences noted in Section~\ref{qc} also hold in this general setting. 

\subsection{Understanding the dependence on $n$}\label{dpn} 
Since the motivation of building our convergence rate analysis is to show that \rsgd behaves better than \sgd, we would definitely hope that our convergence bounds have a better dependence on $T$ compared to the $\mathcal{O}(\frac{1}{T})$ bound for \sgd. In an ideal situation, one may hope for a rate of the form $\mathcal{O}(\frac{1}{T^{1+\delta}})$ with some $\delta>0$. One intuitive criticism toward this goal is evident: if we allow $T<n$, then by setting $n>T^2$, \rsgd is essentially same as \sgd by the birthday paradox. Therefore, a $\mathcal{O}\left(\frac{1}{T^{1+\delta}}\right)$ bound is unlikely to hold.

However, this argument is not rigorous when we require a positive number of epochs to be run (at least one round through all the data). To this end, we provide the following result indicating the impossibility of obtaining $\Oc(\frac{1}{T^{1+\delta}})$ even when $T\geq n$ is required.

\begin{theorem}\label{thm42}
Given the information of $\mu, L, G$. Under the assumption of constant step sizes, no step size choice for \rsgd leads to a convergence rate $o(\nf{1}{T})$ for any $T\geq n$, if we do not allow $n$ to appear in the bound.
\end{theorem}

The key idea to prove Theorem~\ref{thm42} is by constructing a special instance of problem~\eqref{1}. In particular, the following quadratic instance of~\eqref{1} lays the foundation of our proof:
\begin{equation}
  \label{qins}
  f_i(x) = \begin{cases}
    \frac{1}{2} (x-b)' A (x-b) & i\ \text{odd},\\
    \addlinespace[0.4em]
    \frac{1}{2} (x+b)' A (x+b) & i\ \text{even}.\\
  \end{cases}
\end{equation}
Here $(\cdot)'$ denotes the transpose of a vector, $A \in \reals^{d\times d}$ is some positive definite matrix, and $b\in \reals^d$ is some vector. Running \rsgd on~\eqref{qins} leads to a close-formed expression of \rsgd's error. Then by setting $T=n$ (i.e., only running \rsgd for one epoch) and assuming a convergence rate of $o\left(\frac{1}{T}\right)$, we deduce a contradiction by properly setting $A$ and $b$. The detailed proof can be found in Appendix~\ref{proofthm42}. We directly have the following corollary:

\begin{corollary}
Given the information of $\mu, L, G$, under the assumption $T\geq n$ and constant step size, there is no step size choice that leads to a convergence rate $\mathcal{O}\left(\frac{1}{T^{1+\delta}}\right)$ for $\delta>0$.
\end{corollary}

This result indicates that in order to achieve a better dependence on $T$ using  constant step sizes, the bound should either: (i) depend on $n$; (ii) make some stronger assumptions on $T$ being large enough (at least exclude $T=n$); or (iii) leverage a more versatile step size schedule, which could potentially be hard to design and analyze. 

Although Theorem~\ref{thm42} shows that one may not hope (under constant step sizes) for a better dependence on $T$ for \rsgd without an extra $n$ dependence, whether the current dependence on $n$ we have obtained is optimal still requires further discussion. In the special case $n=T$, numerical evidence has shown that \rsgd behaves at least as well as \sgd. However, our bound fails to even show \rsgd converges in this setting. Therefore, it is reasonable to conjecture that a better dependence on $n$ exists. In the following section, we improve the dependence on $n$ under a specific setting. But whether a better dependence on $n$ can be achieved in general remains open.\footnote{Convergence rate with dependence on $n$ also appears in some variance reduction methods (see for instance, ~\citep{johnson2013accelerating, defazio2014saga}). Sample complexity lower bounds has also be shown to depend on $n$ under similar settings, see e.g.,~\citep{arjevani2016dimension}.}

%%% Local Variables:
%%% mode: latex
%%% TeX-master: "paper"
%%% End:

%% file: sparse.tex
\section{Sparse functions}
In the literature on large-scale machine learning, sparsity is a common feature of data. When the data are sparse, each training data point has only a few non-zero features. Under such a setting, each iteration of \sgd only modifies a few dimensions of the decision variables. Some commonly occurring sparse problems include large-scale logistic regression, matrix completion, and graph cuts. 

Sparse data provides a prospective setting under which \rsgd might be powerful. Intuitively, when data are sparse, with-replacement sampling used by \sgd is likely to miss some decision variables, while \rsgd is guaranteed to update all possible decision variables in one epoch. In this section, we show some theoretical results justifying such intuition.

Formally, a sparse finite-sum problem assumes the form
\[F(x) = \frac{1}{n}\sum_{i=1}^n f_i(x_{e_i}),\] where $e_i$ ($1\le i \le n$) denotes a small subset of $\{1, \ldots, d\}$ and $x_{e_i}$ denotes the entries of the vector $x$ indexed by $e_i$. Define the set $E := \{e_i: 1\leq i \leq n\}$. By representing each subset $e_i\subseteq E$ with a node, and considering edges $(e_i, e_j)$ for all $e_i \cap e_j \neq \emptyset$, we get a graph with $n$ nodes. Following the notation in~\citep{recht2011hogwild}, we consider the \emph{sparsity factor} of the graph:
\begin{equation}
  \label{eq:6}
  \rho := \frac{\max\limits_{1\leq i \leq n}\left|\{e_j\in E: e_i \cap e_j \neq \emptyset\}\right|}{n}.
\end{equation}
One obvious fact is $\frac{1}{n}\leq \rho \leq 1$. The statistic \eqref{eq:6} indicates how likely is it that two subsets of indices intersect, which reflects the sparsity of the problem. For a problem with strong sparsity, we may anticipate a relatively small value for $\rho$. We summarize our result with the following theorem:
\begin{theorem}\label{thm52}
Define constant $C = \max\left\{\frac{32}{\mu^2}(L_H LD + 3L_H G), 12(1+\frac{L}{\mu}) \right\}$. So long as $\frac{T}{\log T} > Cn$, with step size $\eta = \frac{8\log T}{ T\mu},$  \rsgd achieves convergence rate: \[\Exp [\norm{x_T - x^*}^2] \leq \mathcal{O}\Bigl(\frac{1}{T^2}+\frac{\rho^2 n^3}{T^3}\Bigr).\]
\end{theorem}
Compared with Theorem~\ref{thm43}, the bound in Theorem~\ref{thm52} depends on the  parameter $\rho$, so we can exploit sparsity to obtain a faster convergence rate. The key to proving Theorem~\ref{thm52} lies in constructing a tighter bound for the error term in the main recursion (see \S\ref{prf}) by including a discount due to sparsity.

We end this section by noting the following simple corollary:
\begin{corollary}
  When $\rho = \mathcal{O}\left(\frac{1}{n}\right)$, there is some constant $C$ only dependent on $\mu$, $L$, $L_H$, $D$, $G$, such that as long as $\frac{T}{\log T} > C n$, for a proper step size, \rsgd achieves convergence rate \[\Exp [\norm{x_T - x^*}^2] \leq \mathcal{O}\Bigl(\frac{1}{T^2} \Bigr).\]
\end{corollary}

%%% Local Variables:
%%% mode: latex
%%% TeX-master: "paper"
%%% End:

%% file: sketch.tex
\section{Proof sketch of Theorem~\ref{thm41}}\label{prf}
In this section we provide a proof sketch for Theorem~\ref{thm41}. The central idea is to establish an inequality 
\begin{equation}
  \Exp\bigl[\norm{x^{t+1}_0 - x^*}^2\bigr] \leq (1-n\gamma\alpha_1)\norm{x^t_0 - x^*}^2 + n\gamma^3 \alpha_2 + n^4\gamma^4 \alpha_3,\label{2}
\end{equation}
where $x^t_0$ and $x^{t+1}_0$ are the beginning and final points of the $t$-th epoch, respectively,  and the randomness is over the permutation $\sigma_t(\cdot)$ of functions in epoch $t$. The constant $\alpha_1$ captures the speed of convergence for the linear convergence part, while $\alpha_2$ and $\alpha_3$ together bound the error introduced by randomness. The underlying motivation for the bound~\eqref{2} is: when the latter two terms depend on the step size $\gamma$ with order at least $3$, then by expanding the recursion over all the epochs, and setting $\gamma = \Oc\bigl(\frac{1}{T}\bigr)$, we can obtain a convergence of $\Oc\bigl(\frac{1}{T^2}\bigr)$. 

By the definition of the \rsgd update and simple calculations, we have the following key equality for one epoch of \rsgd:
\begin{small}
\begin{equation*}
  %\label{eq:7}
  \begin{split}
    \norm{x^{t+1}_0 - x^*}^2 = \norm{x^t_0 - x^*}^2 - \underbrace{2\gamma\langle x^t_0 - x^*, n\nabla F(x^t_0) \rangle}_{A^t_1} - \underbrace{2\gamma\langle x^t_0 - x^*, R^t\rangle}_{A^t_2} + \underbrace{2\gamma^2 \norm{n\nabla F(x^t_0)}^2}_{A^t_3} + \underbrace{2\gamma^2 \norm{R^t}^2}_{A^t_4}.
  \end{split}
\end{equation*}
\end{small}%
The idea behind this equality is to split the progress made by \rsgd in a given epoch into two parts: a part that behaves like full gradient descent ($A^t_1$ and $A^t_3$), and a part that captures the effects of random sampling ($A^t_2$ and $A^t_4$). In particular, for a permutation $\sigma_t(\cdot)$, $R^t$ denotes the gradient error of \rsgd for epoch $t$, i.e., \[R^t = \sum\nolimits_{i=1}^n \nabla f_{\sigma_t(i)} \left(x^t_{i-1}\right) - \sum\nolimits_{i=1}^n \nabla f_{\sigma_t(i)}(x^t_0),\]  
which is a random variable dependent on $\sigma_t(\cdot)$. Thus, the terms $A^t_2$ and $A^t_4$ are also random variables that depend on $\sigma_t(\cdot)$, and require taking expectations. The main body of our analysis involves bounding each of these terms separately.

% Ultimately, Doing so will prove instrumental in efficiently bounding the progress made by \rsgd, which will ultimately yield the bound~\eqref{2}.
% In~\eqref{eq:7},  are the parts of \rsgd that behave like full gradient descent, while  are the parts representing discrepancy between \rsgd and full gradient descent,

The term $A_1^t$ can be easily bounded by exploiting the strong convexity of $F$, using a standard inequality (Theorem 2.1.11 in~\citep{nesterov2013introductory}), as follows
\begin{align}
  A^t_1\geq \frac{2n\gamma}{L+\mu}\norm{\nabla F(x^t_0) }^2 + 2n\gamma\frac{L\mu}{L+\mu}\norm{x^t_0-x^*}^2. \label{bound1}
\end{align}
The first term (gradient norm term) in~\eqref{bound1} is used to dominate later emerging terms in our bounds on $A^t_2$ and $A^t_3$, while the second term (distance term) in~\eqref{bound1} will be absorbed into $\alpha_1$ in~\eqref{2}.

A key step toward building~\eqref{2} is to bound $\Exp[A^t_2]$, where the expectation is over $\sigma_t(\cdot)$. However, it is not easy to directly bound this term with $\gamma^3 C$ for some constant $C$. Instead, we decompose this term further into three parts: (i) the first part depends on $\norm{x^t_0 - x^*}^2$ (which will be then captured by $\alpha_1$ in~\eqref{2}); (ii) the second part depends on $\norm{\nabla F(x^t_0)}^2$ (which will be then dominated by gradient norm term in $A^t_1$'s bound~\eqref{bound1}); and (iii) the third part has an at least $\gamma^3$ dependence on $\gamma$ (which will be then jointly captured by $\alpha_2$ and $\alpha_3$ in~\eqref{2}). Specifically, by introducing second-order information and somewhat involved analysis, we obtain the following bound for $A_2^t$:
\begin{lemma}\label{pro}
  Over the randomness of the permutation, we have the inequality:
\begin{nonumber}
\begin{align}
-2\gamma \langle x^t_0 - x^*, \Exp[R^t] \rangle &\leq \frac{1}{2}\gamma\mu(n-1)\norm{x^t_0-x^*}^2 + \gamma^2n^2 \norm{\nabla F\left(x^t_0\right)}^2 \\
&\ \ \ \ \ \ \ \ + \gamma^{3} \mu^{-1}n^{2}(n-1)\norm{\Delta}^2 + 2\mu^{-1}\gamma^5L^4G^2n^5.
\end{align}
\end{nonumber}
Where $\Delta = \Exp\limits_{i\neq j} H_i(x^*) \nabla f_j\left(x^*\right)$ with $i,j$ uniformly drawn from $[n]$.
\end{lemma}
Since $x^*$ is the minimizer, we have an elegant bound on the second-order interaction term:
\begin{lemma}
Define $\Delta = \Exp\limits_{i\neq j} H_i(x^*) \nabla f_j\left(x^*\right)$ with $i,j$ uniformly drawn from $[n]$, and $x^*$ is the minimizer of sum function, then \[\norm{\Delta} \leq \tfrac{1}{n-1} LG.\]
\end{lemma}

We tackle $A^t_3$ by dominating it with the gradient norm term of $A^t_1$'s bound~\eqref{bound1}, and finally bound the second permutation dependent term $\Exp[A^t_4]$ using the following lemma.
\begin{lemma}
  For any possible permutation in the $t$-th epoch, we have bound
\[\norm{R^t} \leq  \frac{n\left(n-1\right)}{2}\gamma GL.\]
\end{lemma}
Using this bound, the term $\Exp[A^t_4]$ can be captured by $\alpha_3$ in~\eqref{2}.

Based on the above results, we get a recursive inequality of the form~\eqref{2}. Expanding the recursion and substituting into it the step-size choice ultimately leads to an bound of the form $\mathcal{O}\bigl(\frac{1}{T^2} + \frac{n^3}{T^3}\bigr)$ (see~\eqref{a16} in the Appendix for dependence on hidden constants). The detailed technical steps can be found in Appendix~\ref{th1}.
%\begin{align}\Exp[\norm{x_T - x^*}^2] \leq \frac{(\log T)^2}{T^2}\left(D^2 + 128\frac{L^2G^2}{\mu^4}\right) + \frac{n^3(\log T)^4}{T^3}\left(128\frac{L^2G^2}{\mu^4}\right) + \frac{n^4(\log T)^5}{T^4}\left(2048\frac{L^4G^2}{\mu^6}\right). \label{3}
%\end{align}

%%% Local Variables:
%%% mode: latex
%%% TeX-master: "paper"
%%% End:

%% file: discus.tex
\section{Discussion of results}\label{discuss}
We discuss below our results in more detail, including their implications, strengths, and limitations.

\paragraph{Comparison with \sgd.} It is well-known that under strong convexity \sgd converges with a rate of $\mathcal{O}\bigl(\frac{1}{T}\bigr)$~\citep{rakhlin2012making}. A direct comparison indicates the following fact: \rsgd is provably better than \sgd after $\mathcal{O}(\sqrt{n})$ epochs. This is an acceptable amount of epochs for even some of the largest data sets in current machine learning literature. To our knowledge, this is the \emph{first} result rigorously showing that \rsgd behaves better than \sgd within a reasonable number of epochs. To some extent, this result confirms the belief and observation that \rsgd is the ``correct'' choice in real life, at least when the number of epochs is comparable with $\sqrt{n}$.

\paragraph{Deterministic variant.} When the algorithm is run in a deterministic fashion, i.e., the functions $f_i$ are visited in a fixed order, better convergence rate than \sgd can also be achieved as $T$ becomes large. For instance, a result in~\citep{gurbuzbalaban2015convergence} translates into a $\mathcal{O}\bigl(\frac{n^2}{T^2}\bigr)$ bound for the deterministic case. This directly implies the same bound for \rsgd, since random permutation always has the weaker worst case. But according to this bound, at least $n$ epochs are required for \rsgd to achieve an error smaller than \sgd, which is not a realistic number of epochs in most applications.

\paragraph{Comparison with \gd.} Another interesting viewpoint is by comparing \rsgd with Gradient Descent (\gd). One of the limitations of our result is that we do not  show a regime where \rsgd can be better than \gd. By computing the average for each epoch and running exact \gd on~\eqref{1}, one can get a convergence rate of the form $\Oc(\exp(-T/n))$. This fact shows that our convergence rate for \rsgd is worse than \gd. This comes naturally from the epoch based recursion (\ref{2}) in our proof methodology, since for one epoch the sum of the gradients is only shown to be no worse than a full gradient. It is true that \gd should behave better in long-term as the dependence on $n$ is negligible, and comparing with \gd is not the major goal for this paper. However, being worse than \gd even when $T$ is relatively small indicates that the dependence on $n$ probably can still be improved. It may be worth investigating whether \rsgd can be better than both \sgd and \gd in some regime. However, different techniques may be required.

\paragraph{Epochs required.} It is also a limitation that our bound only holds after a certain number of epochs. Moreover, this number of epochs is dependent on $\kappa$ (e.g., $\mathcal{O}(\kappa)$ epochs for the quadratic case). This limits the interest of our result to cases when the problem is not too ill-conditioned. Otherwise, such a number of epochs will be unrealistic by itself. We are currently not certain whether similar bounds can be proved when allowing $T$ to assume smaller values, or even after only one epoch.

\paragraph{Dependence on $\kappa$.}
It should be noticed that $\kappa$ can be large sometimes. Therefore, it may be informative to view our result in a $\kappa$-dependent form. In particular, we still assume $D$, $L$, $L_H$ are constant, but no longer $\mu$. We use the bound $G\leq \max_i{\norm{\nabla f_i(x^*)} + DL}$ and assume $\max_i{\norm{\nabla f_i(x^*)}}$ is constant. Since $\kappa = \nf{L}{\mu}$, we now have $\kappa = \Theta(\nf{1}{\mu})$. 
Our results translate into $\kappa$-dependent convergence rates of $ \mathcal{O}(\nf{\kappa^4}{T^2} + \nf{\kappa^4n^3}{T^3} + \nf{\kappa^6n^4}{T^4})$ (see inequalities~\eqref{a16}~\eqref{b12} in the Appendix). The corresponding $\kappa$-dependent sample complexity turns into $\mathcal{O}(\kappa n + \kappa^2\epsilon^{-\nf12} + n\kappa^{\nf{4}{3}}\epsilon^{-\nf13} + n\kappa^{\nf{3}{2}}\epsilon^{-\nf14})$ for quadratic problems, and $\mathcal{O}(\kappa^2 n + \kappa^2\epsilon^{-\nf12} + n\kappa^{\nf{4}{3}}\epsilon^{-\nf13} + n\kappa^{\nf{3}{2}}\epsilon^{-\nf14})$ for strongly convex ones.

At first sight, the dependence on $\kappa$ in the convergence rate may seem relatively high. However, it is important to notice that our sample complexity's dependence on $\kappa$ is actually \emph{better} than what is known for \sgd. A $\Oc(\frac{4G^2}{T\mu^2})$ convergence bound for \sgd has long been known~\cite{rakhlin2012making}, which translates into a $\mathcal{O}(\frac{\kappa^2}{\epsilon})$, $\kappa$-dependent sample complexity in our notation. Although better $\kappa$ dependence has been shown for $F(x_T) - F(x^*) < \epsilon$ (see e.g.,~\citep{hazan2014beyond}), no better dependence has been shown for $\Exp[\norm{x_T - x^*}^2]<\epsilon$ as far as we know. Furthermore, according to~\cite{nemirovskii1983problem}, the lower bound to achieve $F(x_T) - F(x^*) < \epsilon$ for strongly convex $F$ using stochastic gradients is $\Omega( \nf{\kappa}{\epsilon})$. Translating this into the sample complexity to achieve $\Exp[\norm{x_T - x^*}^2]<\epsilon$  is likely to introduce another $\kappa$ into the bound. Therefore, it is reasonable to believe that $\mathcal{O}(\nf{\kappa^2}{\epsilon})$ is the best sample complexity one can get for \sgd (which is worse than \rsgd), to achieve $\Exp[\norm{x_T - x^*}^2]<\epsilon$.
 
\paragraph{Sparse data setting.}
Notably, in the sparse setting (with sparsity factor $\rho = \mathcal{O}\bigl(\frac{1}{n}\bigr)$), the proven convergence rate is strictly better than the $\mathcal{O}\bigl(\frac{1}{T}\bigr)$ rate of \sgd. This result follows the following intuition: when each dimension is only touched by several functions, letting the algorithm to visit every function would avoid missing certain dimensions. For larger $\rho$, similar speedup can be observed. In fact, so long as we have $\rho = o(n^{-\nf{1}{2}})$, the proven bound is better off than \sgd. Such a result confirms the usage of \rsgd under sparse setting.

%%% Local Variables:
%%% mode: latex
%%% TeX-master: "paper"
%%% End:

%% file: extension.tex
\section{Extensions}
\label{sec:conc}
In this section, we provide some further extensions before concluding with some open problems.

\vspace*{-8pt}
\subsection{\rsgd for nonconvex optimization}\label{lpcondition}
The first extension that we discuss is to nonconvex finite sum problems. In particular, we study \rsgd applied to functions satisfying the \emph{Polyak-\L{}ojasiewicz} condition (also known as gradient dominated functions): \[\frac{1}{2} \norm{\nabla F(x)}^2 \geq \mu (F(x) - F^*),\ \ \ \  \forall x.\]
Here $\mu>0$ is some real number, $F^*$ is the minimal function value of $F(\cdot)$. Strongly convexity is a special situation of this condition with $\mu$ being the strongly convex parameter. One important implication of this condition is that every stationary point is a global minimum. However function $F$ can be non-convex under such setting. Also, it doesn't imply a unique minimum of the function. 

This setting was proposed and analyzed in~\citep{polyak1963gradient}, where a linear convergence rate for \gd was shown. Later, many other optimization methods have been proven efficient under this condition (see~\citep{nesterov2006cubic} for second order methods and~\citep{reddi2016stochastic} for variance reduced gradient methods). Notably, \sgd can be proven to converge with rate $\Oc(\nicefrac{1}{T})$ under this setting (see appendix for a proof).

Assume each component function $f_i$ being $L$ Lipschitz continuous, and the average function $F(x)$ satisfying the Polyak-\L ojasiewicz condition with some constant $\mu$. We have the following extension of our previous result:

\begin{theorem}\label{thpl}
Under the Polyak-\L ojasiewicz condition, define condition number $\kappa = L/\mu$. So long as $\frac{T}{\log T} > 16\kappa^2 n$,  with step size $\eta = \frac{2\log T}{ T\mu}$, \rsgd achieves convergence rate: \[\Exp [\norm{x_T - x^*}^2] \leq \mathcal{O}\Bigl(\frac{1}{T^2} + \frac{n^3}{T^3}\Bigr).\]
\end{theorem}

\subsection{\rsgd for convex problems}\label{genconv}
An important extension of \rsgd is to the general (smooth) convex case without assuming strong convexity. There are no previous results on the convergence rate of \rsgd in this setting that show it to be faster than \sgd. The only result we are aware of is by~\citet{shamir2016without}, who shows \rsgd to be not worse than \sgd in the general (smooth) convex setting. We extend our results to the general convex case, and show a convergence rate that is possibly faster than \sgd, albeit only up to constant terms.

We take the viewpoint of gradients with errors, and denote the difference between component gradient and full gradient as the error: \[\nabla F(x) - \nabla f_i(x) = e_i(x).\] Different assumptions bounding the error term $e_i(x)$ have been studied in optimization literature. We assume that there is a constant $\delta$ that bound the norm of the gradient  error: \[\norm{e_i(x)} \leq \delta,\ \ \ \ \forall x.\] Here $i$ is any index and $x$ is any point in domain. Obviously, $\delta \leq 2G$, with $G$ being the gradient norm bound as before.\footnote{Another common assumption is when the variance of the gradient (i.e., $\Exp[\norm{e_i(x)}^2]$) is bounded. We made the more rigorous assumption here for ease of a simpler analysis. However, there is at most an extra $\sqrt{n}$ term difference between these two assumptions due to the finite sum structure.} 

\begin{theorem}\label{ext1}
Assume $\Delta = \Exp\limits_{i\neq j} H_i(x^*) \nabla f_j\left(x^*\right)$ with $i,j$ uniformly drawn from $[n]$, $x^*$ is an arbitrary minimizer of $F$.
Set stepsize \[\gamma = \min\left\{\frac{1}{16nL}, \sqrt{\frac{D}{Tn\left(\norm{\Delta}+L_H LD^2 + 2L_HDG\right)}}, \left(\frac{D}{Tn^2L^2\delta}\right)^\frac{1}{3}, \left(\frac{1}{Tn^3L^4}\right)^\frac{1}{4} \right\}.\] Assume $\bar{x} = \frac{n}{T} \sum_{i=1}^{T/n} x^i_0$ being the average of epoch ending points of \rsgd. Then there is 
\[F(\bar{x}) - F(x^*) \leq \frac{2D\sqrt{nD\left(\norm{\Delta}+L_H LD^2 + 2L_HDG\right)}}{\sqrt{T}} + \Oc\left(\left(\frac{n}{T}\right)^\frac{2}{3}\delta^\frac{1}{3} + \left(\frac{n}{T}\right)^\frac{3}{4}\right).\]
\end{theorem}

We have some discussion of this result:

Firstly, it is interesting to see what happens asymptotically. We can observe three levels of possible asymptotic (ignore $n$) convergence rates for \rsgd from this theorem: (1) In the most general situation, it converges as $\Oc(\nicefrac{1}{\sqrt{T}})$; (2) when the functions are quadratic (i.e., $L_H=0$) and locally the variance vanishes (i.e., $\Delta = 0$), it converges as $\Oc(\nicefrac{1}{T^{2/3}})$; (3) when the functions are quadratic (i.e., $L_H=0$) and globally the variance vanishes (i.e., $\delta = 0$),  it converges as $\Oc(\nicefrac{1}{T^{3/4}})$. 

Secondly, we should notice that there is a known convergence rate of $\Oc(\nicefrac{DG}{\sqrt{T}})$ for \sgd. Also, we can further bound $\norm{\Delta}$ with $\frac{LG}{n-1}$. Therefore, when $D$ is relatively small and quadratic functions (i.e., $L_H=0$), our bound translates into form of $\Oc(\nicefrac{1}{\sqrt{T}} + \nicefrac{n^{2/3}}{T^{2/3}})$, with constant in front of $\nicefrac{1}{\sqrt{T}}$ possibly smaller than \sgd by constant in certain parameter space. 

One obvious limitation of this result is: when globally there is no variance of gradients, it fails to recover the $\Oc(\nicefrac{1}{T})$ rate of \gd. This indicates the possibility of tighter bounds using more involved analysis. We leave this possibility (either improving upon the $\nicefrac{1}{\sqrt{T}}$ dependence on $T$ under existence of noise, or recovering $\nicefrac{1}{T}$ when there is no noise) as an open question.

\subsection{Vanishing variance}
Our previous results show that \rsgd converges faster than \sgd after a certain number of epochs. However, one may want to see whether it is possible to show faster convergence of \rsgd after only one epoch, or even within one epoch. In this section, we study a specialized class of strongly convex problems where \rsgd has faster convergence rate than \sgd after an \textbf{arbitrary} number of iterations. 

We build our example based on a vanishing variance setting: $\nabla f_i(x^*) = \nabla F(x^*)$ for the optimal point $x^*$. \citet{moulines2011non} show that when $F(x)$ is strongly convex, \sgd converges linearly in this setting. For the construction of our example, we assume a slightly stronger situation: each component function $f_i(x)$ is strongly convex. 

Given $n$ pairs of positive numbers $(\mu_1, L_1), \cdots, (\mu_n, L_n)$ such that $\mu_i \leq L_i$, a dimension $d$ and a point $x^*\in \mathbb{R}^d$, we define a valid problem as a $d$ dimensional finite sum function $F(x) = \sum_{i=1}^n f_i(x)$ where each component $f_i(x)$ is $\mu_i$ strongly convex and has $L_i$ Lipschitz continuous gradient, with some $x^*$ minimizing all functions at the same time (which is equivalent to vanishing gradient). Let $\mathcal{P}$ be the set of all such  problems, called \emph{valid problems} below. For a problem $P\in \mathcal{P}$, let random variable $X_{RS}(T, x_0, \gamma, P)$ be the result of running \rsgd from initial point $x_0$ for $T$ iterations with step size $\gamma$ on problem $P$. Similarly, let $X_{SGD}(T, x_0, \gamma, P)$ be the result of running \sgd from initial point $x_0$ for $T$ iterations with step size $\gamma$ on problem $P$.

We have the following result on the worst-case convergence rate of \rsgd and \sgd:
\begin{theorem}\label{last}
Given $n$ pairs of positive numbers $(\mu_1, L_1), \cdots, (\mu_n, L_n)$ such that $\mu_i \leq L_i$,  a dimension $d$, a point $x^*\in \mathbb{R}^d$ and an initial set $D_R(x^*) = \{x\in \mathbb{R}^d: \norm{x - x^*}_2\leq R\}$. Let $\mathcal{P}$ be the set of valid problems. For step size $\eta\leq\min\limits_i\{\frac{2}{L_i+\mu_i}\}$ and any $T\geq 1$, there is
\[\max\limits_{P\in \mathcal{P}, x_0 \in D_R(x^*)} \Exp[\norm{X_{RS}(T, x_0, \gamma, P) - x^*}^2] \leq \max\limits_{P\in \mathcal{P}, x_0 \in D_R(x^*)}\Exp[ \norm{X_{SGD}(T, x_0, \gamma, P)-x^*}^2].\]
\end{theorem}

This theorem indicates that \rsgd has a better worst-case convergence rate than \sgd after an arbitrary number of iterations under this noted setting.

\section{Conclusion and open problems}
A long-standing problem in the theory of stochastic gradient descent (\sgd) is to prove that \rsgd converges faster than the usual with-replacement \sgd. In this paper, we provide the first non-asymptotic convergence rate analysis for \rsgd. We show in particular that after $\mathcal{O}(\sqrt{n})$ epochs, \rsgd behaves strictly better than \sgd under strong convexity and second-order differentiability. The underlying introduction of dependence on $n$ into the bound plays an important role toward a better dependence on $T$. We further improve the dependence on $n$ for sparse data settings, showing \rsgd's advantage in such situations. 

An important open problem remains: how (and to what extent) can we improve the bound such that \rsgd can be shown to be better than \sgd for smaller $T$. A possible direction is to improve the $n$ dependence arising in our bounds, though different analysis techniques may be required. It is worth noting that for some special settings, this improvement can be achieved. (For example in the setting of Theorem~\ref{last}, \rsgd is shown better than \sgd for any number of iterations.) However, showing \rsgd converges better in general, remains open.

%%% Local Variables:
%%% mode: latex
%%% TeX-master: "paper"
%%% End:

%% file: apx1.tex
\allowdisplaybreaks
%\section{Appendix}
\section{Proof of Theorem \ref{thm41}} \label{th1}
\begin{proof}
Assume $T = nl$ where $l$ is positive integer. Notate $x^t_i$ as the $i$th iteration for $t$th epoch. There is $x_0^1 = x_0$, $x^t_n = x^{t+1}_0$, $x^l_n = x_T$. Assume the permutation used in $t$th epoch is $\sigma_t\left(\cdot\right)$. Define error term \[R^t = \sum_{i=1}^n \nabla f_{\sigma_t\left(i\right)} \left(x^t_{i-1}\right) - \sum_{i=1}^n \nabla f_{\sigma_t\left(i\right)}\left(x^t_0\right).\]

For one epoch of \rsgd, We have the following inequality
\begin{align}
\norm{x^{t}_n - x^*}^2 &= \norm{x^t_0 - x^*}^2 - 2\gamma \left\langle x^t_0 - x^*, \sum_{i=1}^n \nabla f_{\sigma_t\left(i\right)} \left(x^t_{i-1}\right) \right\rangle + \gamma^2 \norm{\sum_{i=1}^n \nabla f_{\sigma_t\left(i\right)} \left(x^t_{i-1}\right)}^2\nonumber\\
&= \norm{x^t_0 - x^*}^2 - 2\gamma\left\langle x^t_0 - x^*, n\nabla F\left(x^t_0\right) \right\rangle - 2\gamma\left\langle x^t_0 - x^*, R^t\right\rangle + \gamma^2 \norm{n\nabla F\left(x^t_0\right) + R^t}^2 \nonumber\\
&\leq \norm{x^t_0-x^*}^2 - 2n\gamma\left[\frac{L\mu}{L+\mu}\norm{x^t_0 - x^*}^2 + \frac{1}{L+\mu} \norm{\nabla F\left(x^t_0\right)}^2\right] \nonumber\\
&\ \ \ \ \ \ \ \ -2\gamma\left\langle x^t_0 - x^*, R^t \right\rangle  + 2\gamma^2 n^2 \norm{\nabla F\left(x^t_0\right)}^2 + 2\gamma^2 \norm{R^t}^2\nonumber\\
&= \left(1-2n\gamma\frac{L\mu}{L+\mu}\right)\norm{x^t_0 - x^*}^2 -\left(2n\gamma\frac{1}{L+\mu}-2\gamma^2n^2\right) \norm{\nabla F\left(x^t_0\right)}^2\nonumber \\
&\ \ \ \ \ \ \ \ - 2\gamma\left\langle x^t_0 - x^*, R^t \right\rangle + 2\gamma^2 \norm{R^t}^2, \label{a1}
\end{align}
where the inequality is due to Theorem 2.1.11 in~\cite{nesterov2013introductory}.

Take the expectation of~\eqref{a1} over randomness of permutation $\sigma_t\left(\cdot\right)$, we have \begin{align}\Exp\left[\norm{x^t_n - x^*}^2\right] &\leq \left(1-2n\gamma\frac{L\mu}{L+\mu}\right)\norm{x^t_0 - x^*}^2 -\left(2n\gamma\frac{1}{L+\mu}-2n^2\gamma^2\right) \norm{\nabla F\left(x^t_0\right)}^2 \nonumber\\
&\ \ \ \ \ \ \ \ -2\gamma \left\langle x^t_0 -x^*, \Exp\left[R^t\right] \right\rangle + 2\gamma^2\Exp\left[\norm{R^t}^2\right]. \label{a2} \end{align}

What remains to be done is to bound the two terms with $R^t$ dependence. Firstly, we give a bound on the norm of $R^t$:
\begin{align}
\norm{R^t} &= \norm{\sum_{i=1}^n \nabla f_{\sigma_t\left(i\right)} \left(x^t_{i-1}\right) - \sum_{i=1}^n \nabla f_{\sigma_t\left(i\right)}\left(x^t_0\right)}\nonumber\\
&\leq \sum_{i=1}^n \norm{\nabla f_{\sigma_t\left(i\right)} \left(x^t_{i-1}\right) - \nabla f_{\sigma_t\left(i\right)}\left(x^t_0\right)}\nonumber\\
&= \sum_{i=1}^n \norm{\sum_{j=1}^{i-1} \left(\nabla f_{\sigma_t\left(i\right)} \left(x^t_{j}\right) - \nabla f_{\sigma_t\left(i\right)} \left(x^t_{j-1}\right)\right) }\nonumber\\
&\leq \sum_{i=1}^n \sum_{j=1}^{i-1} \norm{ \nabla f_{\sigma_t\left(i\right)} \left(x^t_{j}\right) - \nabla f_{\sigma_t\left(i\right)} \left(x^t_{j-1}\right) }\nonumber\\
&\leq \sum_{i=1}^n \sum_{j=1}^{i-1} L \norm{x^t_{j} - x^t_{j-1}}\nonumber\\
&= \sum_{i=1}^n \sum_{j=1}^{i-1} L \norm{-\gamma \nabla f_{\sigma_t\left(j\right)}\left(x^t_{j-1}\right)}\nonumber\\
&\leq \sum_{i=1}^n \sum_{j=1}^{i-1} L\gamma G\nonumber\\
&= \frac{n\left(n-1\right)}{2}\gamma GL,\nonumber
\end{align}
where the first and second inequality is by triangle inequality of vector norm, the third inequality is by definition of $L$, the fourth inequality is by definition of $G$. By this result, we have 
\begin{align}
\Exp\left[\norm{R^t}^2\right] \leq \frac{n^4}{4}\gamma^2 G^2 L^2. \label{a3}
\end{align}

For the $\Exp\left[R^t\right]$ term, we need more careful bound. Since the Hessian is constant for quadratic functions, we use $H_i$ to denote the Hessian matrix of function $f_i(\cdot)$. We begin with the following decomposition:
\begin{align}
R^t &= \sum_{i=1}^n \left[\nabla f_{\sigma_t\left(i\right)} \left(x^t_{i-1}\right) - \nabla f_{\sigma_t\left(i\right)}\left(x^t_{0}\right)\right] \nonumber\\
&= \sum_{i=1}^n \left[H_{\sigma_t\left(i\right)}\left(x^t_{i-1} - x^t_0\right)\right] \nonumber\\
&= \sum_{i=1}^n \left\{H_{\sigma_t\left(i\right)} \sum_{j=1}^{i-1}\left[-\gamma \nabla f_{\sigma_t\left(j\right)} \left(x^t_{j-1}\right)\right]\right\} \nonumber\\
&= \sum_{i=1}^n \left\{-\gamma H_{\sigma_t\left(i\right)} \sum_{j=1}^{i-1}\left[\nabla f_{\sigma_t\left(j\right)} \left(x^t_0\right) + \left(\nabla f_{\sigma_t\left(j\right)} \left(x^t_{j-1}\right)-\nabla f_{\sigma_t\left(j\right)} \left(x^t_0\right)\right)\right]\right\} \nonumber\\
&= -\gamma \sum_{i=1}^n \left[H_{\sigma_t\left(i\right)}\sum_{j=1}^{i-1} \nabla f_{\sigma_t\left(j\right)} \left(x^t_0\right)\right] - \gamma \sum_{i=1}^n \left\{H_{\sigma_t\left(i\right)}\sum_{j=1}^{i-1}\left[\nabla f_{\sigma_t\left(j\right)} \left(x^t_{j-1}\right)-\nabla f_{\sigma_t\left(j\right)} \left(x^t_0\right)\right]\right\} \nonumber\\
&= A^t + B^t. \label{a18}
\end{align}
Here we define random variables \[A^t = -\gamma \sum_{i=1}^n \left[H_{\sigma_t\left(i\right)}\sum_{j=1}^{i-1} \nabla f_{\sigma_t\left(j\right)} \left(x^t_0\right)\right],\]
\[B^t = - \gamma \sum_{i=1}^n \left\{H_{\sigma_t\left(i\right)}\sum_{j=1}^{i-1}\left[\nabla f_{\sigma_t\left(j\right)} \left(x^t_{j-1}\right)-\nabla f_{\sigma_t\left(j\right)} \left(x^t_0\right)\right]\right\}.\]
There is 
\begin{align}\Exp\left[A^t\right] = -\frac{n\left(n-1\right)}{2} \gamma \Exp\limits_{i\neq j}\left[ H_{i} \nabla f_{j}\left(x^t_0\right)\right],\label{a4}\end{align}
\begin{align}
\norm{B^t} &\leq \gamma \sum_{i=1}^n H_{\sigma_t\left(i\right)} \sum_{j=1}^{i-1} \norm{\nabla f_{\sigma_t\left(j\right)} \left(x^t_{j-1}\right) - \nabla f_{\sigma_t\left(j\right)} \left(x^t_0\right)}\nonumber\\
&\leq \gamma \sum_{i=1}^n L \sum_{j=1}^{i-1} \left(j-1\right)\gamma GL\nonumber\\
&= \gamma^2 L^2 G \sum_{i=1}^n \frac{\left(i-1\right)\left(i-2\right)}{2}\nonumber\\
&\leq \frac{1}{2} \gamma^2L^2Gn^3.\label{a5}
\end{align}
Using~\eqref{a18} and ~\eqref{a4}, we can decompose the inner product of $x^t_0 - x^*$ and $\Exp\left[R^t\right]$ into:
\begin{align}
-2\gamma \left\langle x^t_0 - x^*, \Exp\left[R^t\right] \right\rangle &= -2\gamma \left\langle x^t_0 - x^*, \Exp\left[A^t\right] + \Exp\left[B^t\right] \right\rangle\nonumber\\
&= -2\gamma \left\langle x^t_0 - x^*, \Exp\left[A^t\right] \right\rangle - 2\gamma \left\langle x^t_0 - x^*, \Exp\left[B^t\right] \right\rangle\nonumber\\
&= \gamma^2 n\left(n-1\right) \left\langle x^t_0 - x^*, \Exp\limits_{i\neq j} H_{i} \nabla f_{j}\left(x^t_0\right) \right\rangle -2\gamma \left\langle x^t_0 - x^*, \Exp\left[B^t\right] \right\rangle. \label{a6}
\end{align}

For the first term in~\eqref{a6}, there is
\begin{align}
&\ \ \ \ \gamma^2 n\left(n-1\right) \left\langle x^t_0 - x^*, \Exp\limits_{i\neq j} H_i \nabla f_j\left(x^t_0\right)\right\rangle\nonumber\\
&= \gamma^2 n\left(n-1\right) \left\langle x^t_0 - x^*, \Exp\limits_{i\neq j} H_i \left[\nabla f_j\left(x^t_0\right) - \nabla f_j\left(x^*\right)\right]\right\rangle + \gamma^2 n\left(n-1\right) \left\langle x^t_0 - x^*, \Exp\limits_{i\neq j} H_i \nabla f_j\left(x^*\right)\right\rangle\nonumber\\
&\leq \gamma^2n^2 \left\langle x^t_0 - x^*, \Exp\limits_{i, j}H_iH_j\left(x^t_0-x^*\right)\right\rangle + \gamma^2n\left(n-1\right)\left[\frac{\lambda_1}{2}\norm{x^t_0 - x^*}^2 + \frac{1}{2\lambda_1} \norm{\Delta}^2\right]\nonumber\\
&\leq \gamma^2n^2 \norm{\nabla F\left(x^t_0\right)}^2 + \frac{1}{4}\gamma\mu\left(n-1\right)\norm{x^t_0-x^*}^2 + \gamma^{3}\mu^{-1} n^{2}\left(n-1\right)\norm{\Delta}^2. \label{a7}
\end{align}
Here we introduce variable $\Delta = \Exp_{i\neq j} \left[H_i \nabla f_j(x^*)\right]$ for simplicity of notation, with $i, j$ uniformly sampled from all pairs of different indices. The first inequality is by $\left\langle x^t_0 - x^*, H_iH_i\left(x^t_0 - x^*\right)\right\rangle\geq 0$ and AM–GM inequality, where $\lambda_1$ is any positive number. The second inequality comes from noticing that $\Exp\limits_{i, j}H_iH_j = H^2$ (with $i, j$ uniformly sampled from all pairs of indices), and let $\lambda_1 = \frac{1}{2}\mu\gamma^{-1} n^{-1}$. 

For the second term in~\eqref{a6}, we use the bound
\begin{align}
-2\gamma \left\langle x^t_0 - x^*, \Exp\left[B^t\right]\right\rangle &\leq 2\gamma\left[\frac{\lambda_2}{2}\norm{x^t_0- x^*}^2 + \frac{1}{2\lambda_2}\norm{\Exp\left[B^t\right]}^2\right]. \label{a8}
\end{align}
Set $\lambda_2 = \frac{1}{4}\mu\left(n-1\right)$ in~\eqref{a8} and using~\eqref{a5}, there is 
\begin{align}
-2\gamma \left\langle x^t_0 - x^*, \Exp\left[B^t\right]\right\rangle &\leq \frac{1}{4}\gamma\mu\left(n-1\right)\norm{x^t_0- x^*}^2 + 4\gamma\mu^{-1}\left(n-1\right)^{-1}\norm{\Exp\left[B^t\right]}^2\nonumber\\
&\leq \frac{1}{4}\gamma\mu\left(n-1\right)\norm{x^t_0- x^*}^2 + \mu^{-1}\left(n-1\right)^{-1} \gamma^5L^4G^2n^6\nonumber\\
&\leq  \frac{1}{4}\gamma\mu\left(n-1\right)\norm{x^t_0- x^*}^2 + 2\mu^{-1}\gamma^5L^4G^2n^5. \label{a9}
\end{align}
Substituting~\eqref{a7} and~\eqref{a9} back to~\eqref{a6}, we get 
\begin{align}
-2\gamma \left\langle x^t_0 - x^*, \Exp\left[R^t\right] \right\rangle &\leq  \gamma^2n^2 \norm{\nabla F\left(x^t_0\right)}^2 + \frac{1}{2}\gamma\mu\left(n-1\right)\norm{x^t_0-x^*}^2 \nonumber\\
&\ \ \ \ \ \ \ \ \ \ \ \ \ \ \ \ + \gamma^{3} \mu^{-1}n^{2}\left(n-1\right)\norm{\Delta}^2 + 2\mu^{-1}\gamma^5L^4G^2n^5. \label{a10}
\end{align}
The next step requires to bound the $\norm{\Delta}$ term. Toward this end, we use the following important fact:
\begin{align}
\norm{\Delta} &= \norm{\Exp\limits_{i\neq j} H_i \nabla f_j\left(x^*\right)}\nonumber\\
&= \norm{\frac{1}{n\left(n-1\right)} \sum_{i\neq j} H_i \nabla f_j\left(x^*\right)}\nonumber\\
&= \norm{\frac{-1}{n\left(n-1\right)}\sum_{i}H_i\nabla f_i\left(x^*\right)}\nonumber\\
&= \frac{1}{n-1} \norm{\Exp_i\left[H_i \nabla f_i\left(x^*\right)\right]}\nonumber\\
&\leq \frac{1}{n-1} LG. \label{a11}
\end{align}
This fact captures the importance of randomly drawing a permutation instead of using a fixed one. Substituting~\eqref{a3}~\eqref{a10} back to~\eqref{a2} and using~\eqref{a11} , we finally get a recursion bound for one epoch:
\begin{align}
&\ \ \ \ \Exp\norm{x^t_n - x^*}^2  \nonumber\\
&\leq \left(1-2n\gamma\frac{L\mu}{L+\mu} +\frac{1}{2}\gamma\mu\left(n-1\right)\right)\norm{x^t_0 - x^*}^2-\left(2n\gamma\frac{1}{L+\mu}-3\gamma^2n^2\right) \norm{\nabla F\left(x^t_0\right)}^2 \nonumber\\
&\ \ \ \ \ \ \ \ + \gamma^{3} \mu^{-1}n^{2}\left(n-1\right)\norm{\Delta}^2 + 2\mu^{-1}\gamma^5L^4G^2n^5+ \frac{1}{2}n^4\gamma^4G^2L^2 \nonumber\\
&\leq \left(1-2n\gamma\frac{L\mu}{L+\mu} +\frac{1}{2}\gamma\mu\left(n-1\right)\right)\norm{x^t_0 - x^*}^2-\left(2n\gamma\frac{1}{L+\mu}-3\gamma^2n^2\right) \norm{\nabla F\left(x^t_0\right)}^2  \nonumber\\
&\ \ \ \ \ \ \ \ +2\gamma^{3} \mu^{-1}nL^2G^2 + 2\mu^{-1}\gamma^5L^4G^2n^5+ \frac{1}{2}n^4\gamma^4G^2L^2\label{a12}
\end{align}

Now assume \[n\gamma\frac{L\mu}{L+\mu}>\frac{1}{2}\gamma\mu\left(n-1\right),\] and \[2n\gamma\frac{1}{L+\mu}-3\gamma^2n^2>0,\] which we call assumption $1$ and assumption $2$, \eqref{a12} can be further turned into: 
\begin{align}
\Exp\left[\norm{x^t_n - x^*}^2\right] \leq \left(1-n\gamma\frac{L\mu}{L+\mu}\right) \norm{x^t_0 - x^*}^2 + \gamma^3nC_1 + \gamma^{5}n^{5}C_2 + \gamma^4n^4C_3, \label{a13}
\end{align}
where $C_1 = 2\mu^{-1}L^2G^2$, $C_2 = 2\mu^{-1}L^4G^2$, $C_3 = \frac{1}{2}G^2L^2$.
Now assume $n\gamma\frac{L\mu}{L+\mu}<1$, which we call assumption $3$. Expanding~\eqref{a13} over all epochs leads to a final bound of \rsgd:
\begin{align}
\Exp\left[\norm{x_T - x^*}^2\right] \leq \left(1-n\gamma\frac{L\mu}{L+\mu}\right)^{\frac{T}{n}}\norm{x_0 - x^*}^2 + \frac{T}{n} \left(\gamma^3nC_1 + \gamma^{5}n^{5}C_2 + \gamma^4n^4C_3\right). \label{a14}
\end{align}
Not substituting $\gamma = \frac{4\log T }{T\mu}$ into~\eqref{a14}, we have:
\begin{align}
\Exp\left[\norm{x_T - x^*}^2\right] &\leq \left(1-\frac{2n\log T}{T}\right)^{\frac{T}{2n\log T} 2\log T}\norm{x_0 - x^*}^2 + \frac{T}{n} \left(\gamma^3nC_1 + \gamma^{5}n^{5}C_2 + \gamma^4n^4C_3\right) \nonumber\\
&\leq \frac{1}{T^2}\norm{x_0-x^*}^2 + \frac{1}{T^2}\left(\log T\right)^3 C_4 + \frac{n^3}{T^3}\left(\log T\right)^4 C_5 + \frac{n^4}{T^4}\left(\log T\right)^5 C_6, \label{a15}
\end{align}
where $C_4 = \frac{64C_1}{\mu^3}$, $C_5 = \frac{256C_3}{\mu^{4}}$, $C_6 = \frac{1024C_2}{\mu^{5}}$. The first inequality uses the fact that \[ n\frac{4\log T }{T\mu} \frac{L\mu}{L+\mu} \geq \frac{2n\log T}{T}.\] The second inequality comes from $\left(1-x\right)^{\frac{1}{x}} \leq \frac{1}{e}$ for $0<x<1$.
Obviously, \eqref{a15} is a result of the form $\mathcal{O}\left(\frac{1}{T^2} + \frac{n^3}{T^3}\right)$. Or in the expanding version with constant dependence, we have
\begin{align}
\Exp\left[\norm{x_T - x^*}^2\right] \leq \frac{\left(\log T\right)^2}{T^2}\left(D^2 + 128\frac{L^2G^2}{\mu^4}\right) + \frac{n^3\left(\log T\right)^4}{T^3}128\frac{L^2G^2}{\mu^4} + \frac{n^4\left(\log T\right)^5}{T^4}2048\frac{L^4G^2}{\mu^6}. \label{a16}
\end{align}

What remains to determine is to satisfy the three assumptions: (1) $n\gamma\frac{L\mu}{L+\mu}>\frac{1}{2}\gamma\mu\left(n-1\right)$, (2) $2n\gamma\frac{1}{L+\mu}-3\gamma^2n^2>0$, and (3) $n\gamma\frac{L\mu}{L+\mu}<1$. 
The first is naturally satisfied since $\frac{L}{L+\mu} \geq \frac{1}{2}$ and $n>n-1$. The second assumption is equivalent to \[\frac{T}{\log T} > 6\left(1+\frac{L}{\mu}\right)n.\] Assumption $3$ is equivalent to \[\frac{T}{\log T} > \frac{4L}{L+\mu} n,\] which is obviously satisfied when \[\frac{T}{\log T} > 4 n.\] So we only need \[\frac{T}{\log T} > 6\left(1+\frac{L}{\mu}\right)n.\] So whenever $\frac{T}{\log T} > 6\left(1+\frac{L}{\mu}\right)n$, the three assumptions hold. Therefore the theorem is proved.
\end{proof}

\section{Proof of Theorem \ref{thm43}} \label{generalproof}
\begin{proof}
The idea is similar to the proof of theorem~\ref{thm41}, with a slightly different analysis on the $R^t$ term capturing the changing Hessian. For any $i$, we use $H_i$ to denote $H_i\left(x^*\right)$. For any vector $v$ not being zero, define vector value directional function \[dir\left(v\right) = \frac{v}{\norm{v}},\] with norm being $\ell_2$ norm. For the convenience of notation, we define $dir\left(\vec{0}\right) = \vec{0}$, where $\vec{0}$ is the zero vector. For any two points $a,b\in \mathbb{R}^d$, and a matrix function $g\left(\cdot\right): \mathbb{R}^d\rightarrow \mathbb{R}^{d\times d}$, define line integral:

\[\int_{a}^b g\left(x\right) dx := \int_0^{\norm{b-a}} g\left(a+t\frac{b-a}{\norm{b-a}}\right) dir\left(b-a\right) dt,\] where the integral on the right hand side is integral of vector valued function over real number interval. This integral represents integrating the matrix values function along the line from $a$ to $b$. Again, define error term \[R^t = \sum_{i=1}^n \nabla f_{\sigma_t\left(i\right)} \left(x^t_{i-1}\right) - \sum_{i=1}^n \nabla f_{\sigma_t\left(i\right)}\left(x^t_0\right).\]

We have the following decomposition for the error term:
\begin{align}
R^t &= \sum_{i=1}^n \left[\nabla f_{\sigma_t\left(i\right)} \left(x^t_{i-1}\right) - \nabla f_{\sigma_t\left(i\right)}\left(x^t_{0}\right)\right]\nonumber\\
&= \sum_{i=1}^n \left[\int_{x^t_0}^{x^t_{i-1}} H_{\sigma_t\left(i\right)}\left(x\right) dx\right]\nonumber\\
&= \sum_{i=1}^n \left[\int_{x^t_0}^{x^t_{i-1}} H_{\sigma_t\left(i\right)} dx\right] +  \sum_{i=1}^n \left[\int_{x^t_0}^{x^t_{i-1}} \left(H_{\sigma_t\left(i\right)}\left(x\right)-H_{\sigma_t\left(i\right)}\right) dx\right]\nonumber\\
&=  \sum_{i=1}^n\left[H_{\sigma_t\left(i\right)} \left(x^t_{i-1} - x^t_0\right)\right] +  \sum_{i=1}^n \left[\int_{x^t_0}^{x^t_{i-1}} \left(H_{\sigma_t\left(i\right)}\left(x\right)-H_{\sigma_t\left(i\right)}\right) dx\right]\nonumber\\
&= \sum_{i=1}^n\left[H_{\sigma_t\left(i\right)} \sum_{j=1}^{i-1}\left(-\gamma \nabla f_{\sigma_t\left(j\right)}\left(x^t_{j-1}\right)\right)\right] +  \sum_{i=1}^n \left[\int_{x^t_0}^{x^t_{i-1}} \left(H_{\sigma_t\left(i\right)}\left(x\right)-H_{\sigma_t\left(i\right)}\right) dx\right]\nonumber\\
&= -\gamma \sum_{i=1}^n \left[H_{\sigma_t\left(i\right)}\sum_{j=1}^{i-1} \nabla f_{\sigma_t\left(j\right)} \left(x^t_0\right)\right] - \gamma \sum_{i=1}^n \left\{H_{\sigma_t\left(i\right)}\sum_{j=1}^{i-1}\left[\nabla f_{\sigma_t\left(j\right)} \left(x^t_{j-1}\right)-\nabla f_{\sigma_t\left(j\right)} \left(x^t_0\right)\right]\right\} \nonumber\\
&\ \ \ \ + \sum_{i=1}^n \left[\int_{x^t_0}^{x^t_{i-1}} \left(H_{\sigma_t\left(i\right)}\left(x\right)-H_{\sigma_t\left(i\right)}\right) dx\right]\nonumber \\
&= A^t + B^t + C^t.\label{b1}
\end{align}
Here we define random variables \[A^t = -\gamma \sum_{i=1}^n \left[H_{\sigma_t\left(i\right)}\sum_{j=1}^{i-1} \nabla f_{\sigma_t\left(j\right)} \left(x^t_0\right)\right],\]
\[B^t = - \gamma \sum_{i=1}^n \left\{H_{\sigma_t\left(i\right)}\sum_{j=1}^{i-1}\left[\nabla f_{\sigma_t\left(j\right)} \left(x^t_{j-1}\right)-\nabla f_{\sigma_t\left(j\right)} \left(x^t_0\right)\right]\right\},\]
\[C^t = \sum_{i=1}^n \left[\int_{x^t_0}^{x^t_{i-1}} \left(H_{\sigma_t\left(i\right)}\left(x\right)-H_{\sigma_t\left(i\right)}\right) dx\right].\]
Compared with quadratic case, $C^t$ is the new term capturing the difference introduced by a changing Hessian. There is 
\begin{align}
\Exp\left[A^t\right] = -\frac{n\left(n-1\right)}{2} \gamma \Exp\limits_{i\neq j}\left[ H_{i} \nabla f_{j}\left(x^t_0\right)\right], \label{b2}
\end{align}
\begin{align}
\norm{B^t} &\leq \gamma \sum_{i=1}^n H_{\sigma_t\left(i\right)} \sum_{j=1}^{i-1} \left(\nabla f_{\sigma_t\left(j\right)} \left(x^t_{j-1}\right) - \nabla f_{\sigma_t\left(j\right)} \left(x^t_0\right)\right)\nonumber\\
&\leq \gamma \sum_{i=1}^n L \sum_{j=1}^{i-1} \left(j-1\right)\gamma GL\nonumber\\
&= \gamma^2 L^2 G \sum_{i=1}^n \frac{\left(i-1\right)\left(i-2\right)}{2}\nonumber\\
&\leq \frac{1}{2} \gamma^2L^2Gn^3. \label{b3}
\end{align}
\begin{align}
\norm{C^t} &\leq \sum_{i=1}^n \left[\int_{0}^{\norm{x^t_{i-1}-x^t_0}} \norm{H_{\sigma_t\left(i\right)}\left(x^t_0+t\frac{x^t_{i-1}-x^t_0}{\norm{x^t_{i-1}-x^t_0}}\right)-H_{\sigma_t\left(i\right)}}dt\right]\nonumber\\
&\leq \sum_{i=1}^n \left[L_H\max\left\{\norm{x^t_{i-1}-x^*}, \norm{x^t_{0}-x^*}\right\} \norm{x^t_{i-1} - x^t_0}\right]\nonumber\\
&\leq n\left[\left(\norm{x^t_0-x^*} + n\gamma G\right)L_H n\gamma G\right]\nonumber\\
&= n^2\gamma L_H G \norm{x^t_0 - x^*} + n^3\gamma^2 L_H G^2. \label{b4}
\end{align}
Using~\eqref{b1}~\eqref{b2}, we can decompose the innerproduct of $x^t_0 - x^*$ and $\Exp\left[R^t\right]$ as following:
\begin{align}
-2\gamma \left\langle x^t_0 - x^*, \Exp\left[R^t\right] \right\rangle &= -2\gamma \left\langle x^t_0 - x^*, \Exp\left[A^t\right] + \Exp\left[B^t\right]+ \Exp\left[C^t\right] \right\rangle\nonumber\\
&= -2\gamma \left\langle x^t_0 - x^*, \Exp\left[A^t\right] \right\rangle - 2\gamma \left\langle x^t_0 - x^*, \Exp\left[B^t\right] \right\rangle- 2\gamma \left\langle x^t_0 - x^*, \Exp\left[C^t\right] \right\rangle\nonumber\\
&= \gamma^2 n\left(n-1\right) \left\langle x^t_0 - x^*, \Exp\limits_{i\neq j} H_{i} \nabla f_{j}\left(x^t_0\right) \right\rangle -2\gamma \left\langle x^t_0 - x^*, \Exp\left[B^t\right] \right\rangle- 2\gamma \left\langle x^t_0 - x^*, \Exp\left[C^t\right] \right\rangle.\label{b5}
\end{align}
For the first term in the~\eqref{b5}, we have further bound:
\begin{align}
&\ \ \ \ \gamma^2 n\left(n-1\right) \left\langle x^t_0 - x^*, \Exp\limits_{i\neq j} H_i \nabla f_j\left(x^t_0\right)\right\rangle\nonumber\\
&= \gamma^2 n\left(n-1\right)\Exp\limits_{i\neq j} \left\langle H_i \left(x^t_0 - x^*\right),  \nabla f_j\left(x^t_0\right) - \nabla f_j\left(x^*\right)\right\rangle + \gamma^2 n\left(n-1\right) \left\langle x^t_0 - x^*, \Exp\limits_{i\neq j} H_i \nabla f_j\left(x^*\right)\right\rangle\nonumber\\
&\leq \gamma^2n^2 \Exp\limits_{i, j} \left\langle \nabla f_i\left(x^t_0\right) - \nabla f_i\left(x^*\right),\nabla f_j\left(x^t_0\right) - \nabla f_j\left(x^*\right) \right\rangle + \gamma^2n\left(n-1\right)\left[\frac{\lambda}{2}\norm{x^t_0 - x^*}^2+ \frac{1}{2\lambda} \norm{\Delta}^2\right] \nonumber\\
&\ \ \ \ \ \ \ \ +\gamma^2 n\left(n-1\right)\Exp\limits_{i\neq j} \left\langle H_i \left(x^t_0 - x^*\right) - \left(\nabla f_i\left(x^t_0\right) - \nabla f_i\left(x^*\right)\right),  \nabla f_j\left(x^t_0\right) - \nabla f_j\left(x^*\right)\right\rangle\nonumber\\
&\leq \gamma^2n^2 \norm{\nabla F\left(x^t_0\right)}^2 + \frac{1}{4}\gamma\mu\left(n-1\right)\norm{x^t_0-x^*}^2 + \gamma^{3}\mu^{-1} n^{2}\left(n-1\right)\norm{\Delta}^2 + \gamma^2n\left(n-1\right)L_H L\norm{x^t_0 - x^*}^3. \label{b6}
\end{align}
Note that here $H_i\left(x^t_0 - x^*\right)$ is the matrix $H_i(x^*)$ times vector $x^t_0 - x^*$, not the Hessian at point $x^t_0 - x^*$. The last inequality is because of 
\begin{align*}
\norm{H_i\left(x^t_0 - x^*\right) - \left(\nabla f_i\left(x^t_0\right) - \nabla f_i\left(x^*\right)\right)} &=\norm{ H_i\left(x^t_0 - x^*\right) - \int_{x^*}^{x^t_0} H_i\left(x\right) dx}\\
&=\norm{\int_{x^*}^{x^t_0} \left(H_i - H_i\left(x\right)\right)dx}\\
&\leq\int_{0}^{\norm{x^t_0-x^*}} \norm{H_i - H_i\left(x^* + t\frac{x^t_0-x^*}{\norm{x^t_0-x^*}}\right)}dt\\
&\leq L_H \norm{x^t_0 - x^*}^2.
\end{align*}
For the second term in~\eqref{b5}, we use the bound
\begin{align}
-2\gamma \left\langle x^t_0 - x^*, \Exp\left[B^t\right]\right\rangle &\leq \frac{1}{4}\gamma\mu\left(n-1\right)\norm{x^t_0- x^*}^2 + 2\mu^{-1}\gamma^5L^4G^2n^5.\label{b7}
\end{align}
For the third term in~\eqref{b5}, we use the bound
\begin{align}
-2\gamma \left\langle x^t_0 - x^*, \Exp\left[C^t\right]\right\rangle &\leq 2\gamma\norm{x^t_0 - x^*}\cdot\left(n^2\gamma L_H G \norm{x^t_0 - x^*} + n^3\gamma^2 L_H G^2\right)\nonumber\\
&= 2n^2\gamma^2L_H G\norm{x^t_0 - x^*}^2 + \gamma^3n^32\norm{x^t_0 - x^*}L_HG^2\nonumber\\
&\leq 3n^2\gamma^2L_H G\norm{x^t_0 - x^*}^2 + \gamma^4n^4G^3L_H. \label{b8}
\end{align}
Substituting~\eqref{b6}~\eqref{b7}~\eqref{b8} back to~\eqref{b5}, we get 
\begin{align}
-2\gamma \left\langle x^t_0 - x^*, \Exp\left[R^t\right] \right\rangle &\leq  \gamma^2n^2 \norm{\nabla F\left(x^t_0\right)}^2 + \frac{1}{2}\gamma\mu\left(n-1\right)\norm{x^t_0-x^*}^2 + \gamma^{3} \mu^{-1}n^{2}\left(n-1\right)\norm{\Delta}^2 \nonumber\\
&\ \ \ \ + 2\mu^{-1}\gamma^5L^4G^2n^5 + \gamma^4n^4G^3L_H + \gamma^2n^2\left(L_H LD + 3L_H G\right)\norm{x^t_0 - x^*}^2. \label{b9}
\end{align}
Substituting~\eqref{b9} to~\eqref{a2}, for one epoch we get recursion bound:
\begin{align}
&\ \ \ \ \Exp\norm{x^t_n - x^*}^2 \nonumber\\
&\leq \left(1-2n\gamma\frac{L\mu}{L+\mu} +\frac{1}{2}\gamma\mu\left(n-1\right)+\gamma^2n^2\left(L_H LD + 3L_H G\right)\right)\norm{x^t_0 - x^*}^2-\left(2n\gamma\frac{1}{L+\mu}-3\gamma^2n^2\right) \norm{\nabla F\left(x^t_0\right)}^2 \nonumber\\
&\ \ \ \ \ \ \ \ + \gamma^{3}\mu^{-1} n^{2}\left(n-1\right)\norm{\Delta}^2 + 2\mu^{-1}\gamma^5L^4G^2n^5 + \gamma^4n^4G^3L_H + \frac{1}{2}n^4\gamma^4G^2L^2. \label{b10}
\end{align}
Now assume \[\frac{3}{2}n\gamma\frac{L\mu}{L+\mu}>\frac{1}{2}\gamma\mu\left(n-1\right)+\gamma^2n^2\left(L_H LD + 3L_H G\right),\] and \[2n\gamma\frac{1}{L+\mu}-3\gamma^2n^2>0,\] which we call assumption $1$ and assumption $2$, \eqref{b10} can be further turned into:
\begin{align}
\Exp\left[\norm{x^t_n - x^*}^2\right] \leq \left(1-\frac{1}{2}n\gamma\frac{L\mu}{L+\mu}\right) \norm{x^t_0 - x^*}^2 + \gamma^3nC_1 + \gamma^{4}n^4C_2 + \gamma^5n^5C_3, \label{b11}
\end{align}
where $C_1 =2\mu^{-1}L^2G^2$, $C_2 =G^3L_H +\frac{1}{2}G^2L^2$, $C_3 = 2\mu^{-1}L^4G^2$.
Further assume $n\gamma\frac{L\mu}{L+\mu}<1$, which we call assumption $3$, expanding~\eqref{b11} over all the epochs we finally get a bound for \rsgd:
\begin{align*}
\Exp\norm{x_T - x^*}^2 \leq \left(1-\frac{1}{2}n\gamma\frac{L\mu}{L+\mu}\right)^{\frac{T}{n}}\norm{x_0 - x^*}^2 + \frac{T}{n} \left(\gamma^3nC_1 + \gamma^{4}n^{4}C_2 + \gamma^5n^5C_3\right).
\end{align*}
Let $\gamma = \frac{8\log T}{T\mu}$, there is 
\begin{align}
\Exp\norm{x_T - x^*}^2 &\leq \left(1-\frac{2n\log T}{T}\right)^{\frac{T}{2n\log T} 2\log T}\norm{x_0 - x^*}^2 + \frac{T}{n} \left(\gamma^3nC_1 + \gamma^{4}n^{4}C_2 + \gamma^5n^5C_3\right)\nonumber\\
&\leq \frac{1}{T^2}\norm{x_0-x^*}^2 + \frac{1}{T^2}\left(\log T\right)^3 C_4  + \frac{n^3}{T^3}\left(\log T\right)^4 C_5 + \frac{n^4}{T^4}\left(\log T\right)^5 C_6, \label{b12}
\end{align}
where $C_4 = \frac{512C_1}{\mu^3}$,  $C_5 = \frac{4096C_2}{\mu^{4}}$, $C_6 = \frac{8^5C_2}{\mu^{5}}$.The second inequality comes from $\left(1-x\right)^{\frac{1}{x}} \leq \frac{1}{e}$ for $0<x<1$.
Obviously, this is a result of the form $\mathcal{O}\left(\frac{1}{T^2} + \frac{n^3}{T^3}\right)$. 

What remains to determine is to satisfy the three assumptions: (1) $\frac{3}{2}n\gamma\frac{L\mu}{L+\mu}>\frac{1}{2}\gamma\mu\left(n-1\right)+\gamma^2n^2\left(L_H LD + 3L_H G\right)$, (2) $2n\gamma\frac{1}{L+\mu}-3\gamma^2n^2>0$, and (3) $n\gamma\frac{L\mu}{L+\mu}<1$. 
The first is satisfied when\[n\gamma\frac{L\mu}{L+\mu}>\frac{1}{2}\gamma\mu\left(n-1\right),\] which is naturally satisfied and \[\frac{1}{2}n\gamma\frac{L\mu}{L+\mu}>\gamma^2n^2\left(L_H LD + 3L_H G\right),\] which is equivalent to \[\frac{T}{\log T} > 16\frac{L+\mu}{L\mu^2}\left(L_H LD + 3L_H G\right)n,\] which is obviously satisfied if we assume \[\frac{T}{\log T} > \frac{32}{\mu^2}\left(L_H LD + 3L_H G\right)n.\] The second assumption is equivalent to \[\frac{T}{\log T} > 12\left(1+\frac{L}{\mu}\right) n.\] Assumption $3$ is equivalent to \[\frac{T}{\log T} > \frac{8L}{L+\mu}n,\] which is satisfied when \[\frac{T}{\log T} > 8n.\] Since $12\left(1+\frac{L}{\mu}\right)>8$, we only need \[\frac{T}{\log T} > \max\left\{ \frac{32}{\mu^2}\left(L_H LD + 3L_H G\right)n, 12\left(1+\frac{L}{\mu}\right) n \right\}.\]  So whenever $\frac{T}{\log T} > \max\left\{\frac{32}{\mu^2}\left(L_H LD + 3L_H G\right), 12\left(1+\frac{L}{\mu}\right) \right\}n$, the three assumptions hold. Therefore the theorem is proved.
\end{proof}

%% file: apx2.tex
\section{Proof of Theorem \ref{thm42}} \label{proofthm42}

\begin{proof}
We only need to show that when $T=n$ (i.e., one epoch is run for each problem) and $n$ is even, no such step size schedule exists. We note the random permutation of this single epoch as $\sigma(\cdot)$. For $n$ even, consider the following quadratic problem: \[F(x) = \frac{1}{n}\sum_{i=1}^n f_i(x),\] where 
\begin{align*}
f_i(x) = \left\{\begin{array}{ll}
                  \frac{1}{2} (x-b)' A (x-b) & i\ odd,\\
                  \\
                  \frac{1}{2} (x+b)' A (x+b) & i\ even,\\
                \end{array}
              \right.
\end{align*}
where $A$ is some $d\times d$ positive definite matrix with minimal eigenvalue $\mu$ and maximal eigenvalue $L$, $b$ is a $d$ dimensional vector. We use $(\cdot)'$ to notate the transpose, so as to distinguish from exponential $T$. The exact value of $A$ and $b$ will be determined later. Obviously, $x^*=0$ is the minimizer. In this setting, we have:
\begin{align}
x_{t}  &= x_{t-1}  - \gamma A(x_{t-1} + (-1)^{\sigma(t)}b)\nonumber\\
&= (I - \gamma A) x_{t-1} -(-1)^{\sigma(t)} \gamma A b. \label{c1}
\end{align}
Expanding~\eqref{c1} over iterations leads to: 
\begin{align}
x_T  &= (I-\gamma A)^T x_0 - \sum_{t=1}^T (-1)^{\sigma(t)} \gamma(I-\gamma A)^{T-t} A b. \label{c2}
\end{align}

Taking expectation of~\eqref{c2} over the randomness of $\sigma$, there is 
\begin{align}
\Exp\left[x_T\right] = (I-\gamma A)^T x_0. \label{c3}
\end{align}

With~\eqref{c2}~\eqref{c3}, we have close-formed expression on the final error:
\begin{align}
\Exp\left[\norm{x_T - x^*}^2\right] &=  \norm{\Exp\left[x_T\right] - x^*}^2 + \Exp\left[\norm{x_T - \Exp\left[x_T\right]}^2\right]\nonumber\\
&= \norm{(I-\gamma A)^T (x_0-x^*)}^2 + \Exp\left[\norm{\sum_{t=1}^T (-1)^{\sigma(t)} \gamma(I-\gamma A)^{T-t} A b}^2\right]. \label{c4}
\end{align}

Assume the eigenvalues of $A$ are $\lambda_1, \lambda_2,\cdots, \lambda_d$, there is an orthogonal basis $e_1, \cdots, e_d$ for $\mathbb{R}^d$ such that $e_k$ is eigenvector of $A$ with eigenvalue $\lambda_k$. We can write \[b = \sum_{i=1}^d b_i e_i.\] Since $\left\langle e_i, e_j\right\rangle = 0$ for $i\neq j$, we can simplify the last term in~\eqref{c4}:
\begin{align}
\norm{\sum_{t=1}^T (-1)^{\sigma(t)} \gamma(I-\gamma A)^{T-t} A b}^2 &= \norm{\sum_{t=1}^T (-1)^{\sigma(t)} \gamma(I-\gamma A)^{T-t} A (\sum_{i=1}^d b_ie_i)}^2\nonumber\\
&= \norm{\sum_{i=1}^d\left[ \sum_{t=1}^T (-1)^{\sigma(t)} \gamma(I-\gamma A)^{T-t} A (b_ie_i)\right]}^2\nonumber\\
&= \norm{\sum_{i=1}^d\left[ \sum_{t=1}^T (-1)^{\sigma(t)} \gamma(1-\gamma \lambda_i)^{T-t} \lambda_i (b_ie_i)\right]}^2\nonumber\\
&= \sum_{i=1}^d\left[ \sum_{t=1}^T (-1)^{\sigma(t)} \gamma(1-\gamma \lambda_i)^{T-t} \lambda_i b_i\right]^2\nonumber\\
&= \gamma^2 \sum_{i=1}^d b_i^2  \lambda_i^2 \left[ \sum_{t=1}^T (-1)^{\sigma(t)} (1-\gamma \lambda_i)^{T-t}\right]^2. \label{c5}
\end{align}

Substituting~\eqref{c5} to~\eqref{c4}, we have 
\begin{align}
\Exp\left[\norm{x_T - x^*}^2\right] &=  \norm{(I-\gamma A)^T (x_0-x^*)}^2 + \gamma^2 \sum_{i=1}^d b_i^2  \lambda_i^2 \Exp\left[\left[ \sum_{t=1}^T (-1)^{\sigma(t)} (1-\gamma \lambda_i)^{T-t}\right]^2\right] \label{c6}
\end{align}

Once again, we can write \[x_0 - x^* = \sum_{i=1}^d a_i e_i.\] Then~\eqref{c6} can simplified as 
\begin{align}
\Exp\left[\norm{x_T - x^*}^2\right] &=  \sum_{i=1}^d (1-\gamma \lambda_i)^{2T} a_i^2 + \gamma^2 \sum_{i=1}^d b_i^2  \lambda_i^2 \Exp\left[\left[ \sum_{t=1}^T (-1)^{\sigma(t)} (1-\gamma \lambda_i)^{T-t}\right]^2\right] \label{c7}
\end{align}

Define random variables $s_t = (-1)^{\sigma(t)}$ for $t=1,\cdots, T$. Then for any index pair $t\neq u$, over randomness of $\sigma$, there is 
\begin{align*}
\Exp\left[s_ts_u\right] &= \frac{2\frac{(\frac{T}{2})(\frac{T}{2}-1)}{2}}{\frac{T(T-1)}{2}} - \frac{(\frac{T}{2})(\frac{T}{2})}{\frac{T(T-1)}{2}}\\
&= -\frac{1}{T-1}.
\end{align*}
Using this fact, we can simplify the last term in~\eqref{c7} as:
\begin{align}
\Exp\left[\left[ \sum_{t=1}^T (-1)^{\sigma(t)} (1-\gamma \lambda_i)^{T-t}\right]^2\right] &= \sum_{t=1}^T (1-\gamma \lambda_i)^{2(T-t)} + \sum_{t\neq u}(1-\gamma \lambda_i)^{2T-t-u}\Exp\left[s_ts_u\right]\nonumber\\
 &= \sum_{t=0}^{T-1} (1-\gamma \lambda_i)^{2t} - \frac{1}{T-1} \sum_{t=0}^{T-1}\sum_{u=0, u\neq t}^{T-1}(1-\gamma \lambda_i)^{t+u}\nonumber\\
 &= \sum_{t=0}^{T-1} (1-\gamma \lambda_i)^{2t}  + \frac{1}{T-1} \sum_{t=0}^{T-1} (1-\gamma \lambda_i)^{2t} - \frac{1}{T-1} \left[\sum_{t=0}^{T-1} (1-\gamma \lambda_i)^{t} \right]^2\nonumber\\
 &= \frac{T}{T-1} \frac{1-(1-\gamma \lambda_i)^{2T}}{1-(1-\gamma \lambda_i)^{2}} - \frac{1}{T-1} \left[\frac{1-(1-\gamma \lambda_i)^T}{\gamma\lambda_i}\right]^2. \label{c8}
\end{align}
For contradiction, we assume for any $T$, there is a $\gamma$ dependent on $T$ such that \begin{align}\Exp\left[\norm{x_T - x^*}^2\right] \leq o(\nicefrac{1}{T}).\label{fasterconvergence}\end{align} Now we determine the specific requirement of $A$ and $b$. The only requirement is: $A$ has at least three different positive eigenvalues $\lambda_1 > \lambda_2 > \lambda_3$ , and $b_i\neq 0$ for any $i$. Furthermore, we assume $a_i\neq 0$ for any $i$. Now for the faster convergence rate~\eqref{fasterconvergence} to hold, from~\eqref{c7} we know there must be \begin{align}(1-\gamma \lambda_i)^{2T} = o(\frac{1}{T}),\label{1faster}\end{align} \begin{align}\gamma^2 \Exp\left[\left[ \sum_{t=1}^T (-1)^{\sigma(t)} (1-\gamma \lambda_i)^{T-t}\right]^2\right] = o(\frac{1}{T}),\label{2faster}\end{align} hold for any $i$. 

However with~\eqref{c8}, we know:
\begin{align}
&\ \ \ \ \gamma^2 \Exp\left[\left[ \sum_{t=1}^T (-1)^{\sigma(t)} (1-\gamma \lambda_i)^{T-t}\right]^2\right] \nonumber\\
&= \gamma^2\left\{ \frac{T}{T-1} \frac{1-(1-\gamma \lambda_i)^{2T}}{1-(1-\gamma \lambda_i)^{2}} - \frac{1}{T-1} \left[\frac{1-(1-\gamma \lambda_i)^T}{\gamma\lambda_i}\right]^2\right\}\nonumber\\
&= \gamma^2\left[\frac{T}{T-1} \frac{1}{1-(1-\gamma \lambda_i)^{2}} - \frac{1}{T-1} \frac{1}{\gamma^2\lambda_i^2}\right] + \gamma^2\left[\frac{T}{T-1} \frac{(1-\gamma \lambda_i)^{2T}}{1-(1-\gamma \lambda_i)^{2}} - \frac{1}{T-1} \frac{-2(1-\gamma \lambda_i)^T+(1-\gamma \lambda_i)^{2T}}{\gamma^2\lambda_i^2}\right].\label{c9}
\end{align}
So by~\eqref{2faster}, there must be~\eqref{c9} is $o(\frac{1}{T})$. We now analyze the terms in~\eqref{c9}. There must be $|1-\gamma\lambda_1|<1$ for convergence, so $|\gamma|$ is no more than $\frac{2}{\lambda_1}$ which is constant. Since~\eqref{1faster}, there is $(1-\gamma \lambda_i)^{T} = o(1) $, so \[\gamma^2\left[- \frac{1}{T-1} \frac{-2(1-\gamma \lambda_i)^T+(1-\gamma \lambda_i)^{2T}}{\gamma^2\lambda_i^2}\right] = o(\frac{1}{T}).\] Again, since $|1-\gamma\lambda_1|<1$, for $i=2,3$ there is \[|\frac{\gamma^2}{2\gamma\lambda_i - \gamma^2\lambda_i^2}|\leq \frac{\frac{2}{\lambda_1}}{(2-\frac{2\lambda_i}{\lambda_1})\lambda_i}\] which is constant. Therefore by~\eqref{1faster}, \[\gamma^2\left[\frac{T}{T-1} \frac{(1-\gamma \lambda_i)^{2T}}{1-(1-\gamma \lambda_i)^{2}}\right] = o(\frac{1}{T})\] for $i=2,3$. So for what remains in~\eqref{c9}, \[\gamma^2\left[\frac{T}{T-1} \frac{(1-\gamma \lambda_i)^{2T}}{1-(1-\gamma \lambda_i)^{2}} - \frac{1}{T-1} \frac{-2(1-\gamma \lambda_i)^T+(1-\gamma \lambda_i)^{2T}}{\gamma^2\lambda_i^2}\right]=o(\frac{1}{T})\] for $i=2,3$. Therefore, 

\[\gamma^2\left[\frac{T}{T-1} \frac{1}{1-(1-\gamma \lambda_i)^{2}} - \frac{1}{T-1} \frac{1}{\gamma^2\lambda_i^2}\right] = o(\frac{1}{T}),\] so \[\gamma \frac{T}{T-1}\frac{1}{\lambda_i(2-\gamma\lambda_i)} = \frac{1}{T-1}\frac{1}{\lambda_i^2}+o(\frac{1}{T}),\] which means \[\frac{\gamma T}{2-\gamma\lambda_i} = \frac{1}{\lambda_i} + o(1).\] Since $2-\frac{2}{\lambda_1}\lambda_i \leq 2-\gamma\lambda_i \leq 2$ for $i=2,3$, there must be \[\sup\lim\limits_{T\rightarrow \infty}\gamma T<C\]  for some $C>0$, so $\gamma\rightarrow 0$ as $T\rightarrow \infty$. Therefore, $(2-\gamma\lambda_i)\rightarrow 2$. So there has to be \[\lim\limits_{T\rightarrow \infty}\gamma T = \frac{2}{\lambda_i}.\] However, this cannot be true for $\lambda_2 \neq \lambda_3$ at the same time, contradiction. As a result, no step size can leads to convergence of $o(\frac{1}{T})$.
\end{proof}

\section{Proof of Theorem \ref{thm52}}
\begin{proof}
The idea is similar to the proof of theorem~\ref{thm43}, with a slightly different analysis on the $R^t$ term adopting the sparsity parameter. For any $i$, we use $H_i$ to denote $H_i(x^*)$. Again, we have the following decomposition for the error term:
\begin{align}
R^t &= \sum_{i=1}^n \left[\nabla f_{\sigma_t(i)} (x^t_{i-1}) - \nabla f_{\sigma_t(i)}(x^t_{0})\right]\nonumber\\
&= \sum_{i=1}^n \left[\int_{x^t_0}^{x^t_{i-1}} H_{\sigma_t(i)}(x) dx\right]\nonumber\\
&= \sum_{i=1}^n \left[\int_{x^t_0}^{x^t_{i-1}} H_{\sigma_t(i)} dx\right] +  \sum_{i=1}^n \left[\int_{x^t_0}^{x^t_{i-1}} (H_{\sigma_t(i)}(x)-H_{\sigma_t(i)}) dx\right]\nonumber\\
&=  \sum_{i=1}^n\left[H_{\sigma_t(i)} (x^t_{i-1} - x^t_0)\right] +  \sum_{i=1}^n \left[\int_{x^t_0}^{x^t_{i-1}} (H_{\sigma_t(i)}(x)-H_{\sigma_t(i)}) dx\right]\nonumber\\
&= \sum_{i=1}^n\left[H_{\sigma_t(i)} \sum_{j=1}^{i-1}(-\gamma \nabla f_{\sigma_t(j)}(x^t_{j-1}))\right] +  \sum_{i=1}^n \left[\int_{x^t_0}^{x^t_{i-1}} (H_{\sigma_t(i)}(x)-H_{\sigma_t(i)}) dx\right]\nonumber\\
&= -\gamma \sum_{i=1}^n \left[H_{\sigma_t(i)}\sum_{j=1}^{i-1} \nabla f_{\sigma_t(j)} (x^t_0)\right] - \gamma \sum_{i=1}^n \left\{H_{\sigma_t(i)}\sum_{j=1}^{i-1}\left[\nabla f_{\sigma_t(j)} (x^t_{j-1})-\nabla f_{\sigma_t(j)} (x^t_0)\right]\right\} \nonumber\\
&\ \ \ \ + \sum_{i=1}^n \left[\int_{x^t_0}^{x^t_{i-1}} (H_{\sigma_t(i)}(x)-H_{\sigma_t(i)}) dx\right] \nonumber\\
&=A^t + B^t + C^t.\label{d1}
\end{align}
Here we define random variables \[A^t = -\gamma \sum_{i=1}^n \left[H_{\sigma_t(i)}\sum_{j=1}^{i-1} \nabla f_{\sigma_t(j)} (x^t_0)\right],\]
\[B^t = - \gamma \sum_{i=1}^n \left\{H_{\sigma_t(i)}\sum_{j=1}^{i-1}\left[\nabla f_{\sigma_t(j)} (x^t_{j-1})-\nabla f_{\sigma_t(j)} (x^t_0)\right]\right\},\]
\[C^t = \sum_{i=1}^n \left[\int_{x^t_0}^{x^t_{i-1}} (H_{\sigma_t(i)}(x)-H_{\sigma_t(i)}) dx\right].\]
This time, we have bounds for these three terms adopting sparsity information:
\begin{align}
\Exp\left[A^t\right] = -\frac{n(n-1)}{2} \gamma \Exp\limits_{i\neq j}\left[ H_{\sigma_t(i)} \nabla f_{\sigma_t(j)}(x^t_0)\right], \label{d2}
\end{align}
\begin{align}
\norm{B^t} &\leq \gamma \sum_{i=1}^n H_{\sigma_t(i)} \sum_{j=1}^{i-1} (\nabla f_{\sigma_t(j)} (x^t_{j-1}) - \nabla f_{\sigma_t(j)} (x^t_0))\nonumber\\
&\leq \gamma \sum_{i=1}^n L \sum_{j=1}^{i-1} \rho n \gamma GL\nonumber\\
&\leq n^3 \gamma^2 \rho GL^2. \label{d3}
\end{align}
\begin{align}
\norm{C^t} &\leq \sum_{i=1}^n \sum_{j=1}^{i-1}\norm{\int_{x^t_{j-1}}^{x^t_j} (H_{\sigma_t(i)}(x)-H_{\sigma_t(i)}) dx}\nonumber\\
&\leq \sum_{i=1}^n \rho n \left[\max\left\{\norm{x^t_{j}-x^*}| 0\leq j \leq i-1\right\} L_H\gamma G\right]\nonumber\\
&\leq \rho n^2\left[(\norm{x^t_0-x^*} + n\gamma G)L_H \gamma G\right]\nonumber\\
&= \rho n^2\gamma L_H G \norm{x^t_0 - x^*} + \rho n^3\gamma^2 L_H G^2. \label{d4}
\end{align}
Here the introduction of $\rho$ in~\eqref{d3} is because: if $f_{\sigma_t(k)} $ and  $f_{\sigma_t(j)} $ depend on disjoint dimensions of variables and $k<j$, then there must be $\nabla f_{\sigma_t(j)} (x^t_k) = \nabla f_{\sigma_t(j)} (x^t_{k-1})$. The introduction of $\rho$ in~\eqref{d4} is similar: if $f_{\sigma_t(i)} $ and  $f_{\sigma_t(j)} $ depend on disjoint dimensions of variables and $j<i$, then there must be $\int_{x^t_{j-1}}^{x^t_j} (H_{\sigma_t(i)}(x)-H_{\sigma_t(i)}) dx=0$.

With~\eqref{d1}~\eqref{d2}, we can decompose the innerproduct of $x^t_0 - x^*$ and $\Exp\left[R^t\right]$ into:
\begin{align}
-2\gamma \left\langle x^t_0 - x^*, \Exp\left[R^t\right] \right\rangle &= -2\gamma \left\langle x^t_0 - x^*, \Exp\left[A^t\right] + \Exp\left[B^t\right]+ \Exp\left[C^t\right] \right\rangle\nonumber\\
&= -2\gamma \left\langle x^t_0 - x^*, \Exp\left[A^t\right] \right\rangle - 2\gamma \left\langle x^t_0 - x^*, \Exp\left[B^t\right] \right\rangle- 2\gamma \left\langle x^t_0 - x^*, \Exp\left[C^t\right] \right\rangle\nonumber\\
&= \gamma^2 n(n-1) \left\langle x^t_0 - x^*, \Exp\limits_{i\neq j} H_{i} \nabla f_{j}(x^t_0) \right\rangle -2\gamma \left\langle x^t_0 - x^*, \Exp\left[B^t\right] \right\rangle- 2\gamma \left\langle x^t_0 - x^*, \Exp\left[C^t\right] \right\rangle.\label{d5}
\end{align}
For the first term in the~\eqref{d5}, there is
\begin{align}
&\ \ \ \ \gamma^2 n(n-1) \left\langle x^t_0 - x^*, \Exp\limits_{i\neq j} H_i \nabla f_j(x^t_0)\right\rangle\nonumber\\
&= \gamma^2 n(n-1)\Exp\limits_{i\neq j} \left\langle H_i (x^t_0 - x^*),  \nabla f_j(x^t_0) - \nabla f_j(x^*)\right\rangle + \gamma^2 n(n-1) \left\langle x^t_0 - x^*, \Exp\limits_{i\neq j} H_i \nabla f_j(x^*)\right\rangle\nonumber\\
&\leq \gamma^2n^2 \Exp\limits_{i, j} \left\langle \nabla f_i(x^t_0) - \nabla f_i(x^*),\nabla f_j(x^t_0) - \nabla f_j(x^*) \right\rangle + \gamma^2n(n-1)\left[\frac{\lambda}{2}\norm{x^t_0 - x^*}^2+ \frac{1}{2\lambda} \norm{\Delta}^2\right] \nonumber\\
&\ \ \ \ \ \ \ \ +\gamma^2 n(n-1)\Exp\limits_{i\neq j} \left\langle H_i (x^t_0 - x^*) - (\nabla f_i(x^t_0) - \nabla f_i(x^*)),  \nabla f_j(x^t_0) - \nabla f_j(x^*)\right\rangle\nonumber\\
&\leq \gamma^2n^2 \norm{\nabla F(x^t_0)}^2 + \frac{1}{4}\gamma\mu(n-1)\norm{x^t_0-x^*}^2 + \gamma^{3}\mu^{-1} n^{2}(n-1)\norm{\Delta}^2 + \gamma^2n(n-1)L_H L\norm{x^t_0 - x^*}^3. \label{d6}
\end{align}
Where the last inequality is because of 
\begin{align*}
\norm{H_i(x^t_0 - x^*) - (\nabla f_i(x^t_0) - \nabla f_i(x^*))} &=\norm{ H_i(x^t_0 - x^*) - \int_{x^*}^{x^t_0} H_i(x) dx}\\
&=\norm{\int_{x^*}^{x^t_0} (H_i - H_i(x))dx}\\
&\leq\int_{0}^{\norm{x^t_0-x^*}} \norm{H_i - H_i\left(x^* + t\frac{x^t_0-x^*}{\norm{x^t_0-x^*}}\right)}dt\\
&\leq L_H \norm{x^t_0 - x^*}^2.
\end{align*}
For the second term in~\eqref{d5}, we use the bound
\begin{align}
-2\gamma \left\langle x^t_0 - x^*, \Exp\left[B^t\right]\right\rangle &\leq \frac{1}{4}\gamma\mu(n-1)\norm{x^t_0- x^*}^2 + 2\mu^{-1}\gamma^5\rho^2L^4G^2n^5. \label{d7}
\end{align}
For the third term in~\eqref{d5}, we use the bound
\begin{align}
-2\gamma \left\langle x^t_0 - x^*, \Exp\left[C^t\right]\right\rangle &\leq 2\gamma\norm{x^t_0 - x^*}\cdot(\rho n^2\gamma L_H G \norm{x^t_0 - x^*} + \rho n^3\gamma^2 L_H G^2)\nonumber\\
&= 2n^2\rho\gamma^2L_H G\norm{x^t_0 - x^*}^2 + \rho\gamma^3n^32\norm{x^t_0 - x^*}L_HG^2\nonumber\\
&\leq (2\rho+1)n^2\gamma^2L_H G\norm{x^t_0 - x^*}^2 + \rho^2\gamma^4n^4G^3L_H\nonumber\\
&\leq 3n^2\gamma^2L_H G\norm{x^t_0 - x^*}^2 + \rho^2\gamma^4n^4G^3L_H. \label{d8}
\end{align}
Substituting~\eqref{d6}~\eqref{d7}~\eqref{d8} back to~\eqref{d5}, we get 
\begin{align}
-2\gamma \left\langle x^t_0 - x^*, \Exp\left[R^t\right] \right\rangle &\leq  \gamma^2n^2 \norm{\nabla F(x^t_0)}^2 + \frac{1}{2}\gamma\mu(n-1)\norm{x^t_0-x^*}^2 + \gamma^{3} \mu^{-1}n^{2}(n-1)\norm{\Delta}^2 \nonumber\\
&\ \ \ \ + 2\mu^{-1}\gamma^5\rho^2L^4G^2n^5 +\rho^2\gamma^4n^4G^3L_H + \gamma^2n^2(L_H LD + 3 L_H G)\norm{x^t_0 - x^*}^2. \label{d9}
\end{align}
Substituting~\eqref{d9} to~\eqref{a2}, we have recursion bound for one epoch:
\begin{align*}
&\ \ \ \ \Exp\norm{x^t_n - x^*}^2 \\
&\leq (1-2n\gamma\frac{L\mu}{L+\mu} +\frac{1}{2}\gamma\mu(n-1)+\gamma^2n^2(L_H LD + 3 L_H G))\norm{x^t_0 - x^*}^2-(2n\gamma\frac{1}{L+\mu}-3\gamma^2n^2) \norm{\nabla F(x^t_0)}^2 \\
&\ \ \ \ \ \ \ \ + \gamma^{3} \mu^{-1}n^{2}(n-1)\norm{\Delta}^2 + \rho^2\gamma^4n^4G^3L_H + 2\mu^{-1}\gamma^5\rho^2L^4G^2n^5 + 2\rho^2 n^4\gamma^4G^2L^2.
\end{align*}
Here the last inequality is because
\begin{align*}
\norm{R^t} &= \norm{\sum_{i=1}^n \nabla f_{\sigma_t(i)} (x^t_{i-1}) - \sum_{i=1}^n \nabla f_{\sigma_t(i)}(x^t_0)}\\
&\leq \sum_{i=1}^n \norm{\nabla f_{\sigma_t(i)} (x^t_{i-1}) - \nabla f_{\sigma_t(i)}(x^t_0)}\\
&= \sum_{i=1}^n \norm{\sum_{j=1}^{i-1} (\nabla f_{\sigma_t(i)} (x^t_{j}) - \nabla f_{\sigma_t(i)} (x^t_{j-1})) }\\
&\leq \sum_{i=1}^n \sum_{j=1}^{i-1} \norm{ \nabla f_{\sigma_t(i)} (x^t_{j}) - \nabla f_{\sigma_t(i)} (x^t_{j-1}) }\\
&\leq n^2 \rho L\gamma G.
\end{align*}
Finally, we again use the fact
 \[\norm{\Delta} \leq \frac{1}{n-1} L G.\]
The remaining process is same as proof of theorem \ref{thm43}, leading to a bound $\mathcal{O}( \frac{1}{T^2} + \frac{\rho^2 n^3}{T^3})$.

\end{proof}

\section{Proof of Theorem~\ref{thpl}} \label{proofPL}
\begin{proof}
The idea is similar to the proof of theorem~\ref{thm43}. For any vector $v$ not being zero, define vector value directional function \[dir\left(v\right) = \frac{v}{\norm{v}},\] with norm being $\ell_2$ norm. For the convenience of notation, we define $dir\left(\vec{0}\right) = \vec{0}$, where $\vec{0}$ is the zero vector. For any two points $a,b\in \mathbb{R}^d$, and a matrix function $g\left(\cdot\right): \mathbb{R}^d\rightarrow \mathbb{R}^{d\times d}$, define line integral:

\[\int_{a}^b g\left(x\right) dx := \int_0^{\norm{b-a}} g\left(a+t\frac{b-a}{\norm{b-a}}\right) dir\left(b-a\right) dt,\] where the integral on the right hand side is integral of vector valued function over real number interval. This integral represents integrating the matrix values function along the line from $a$ to $b$. Again, define error term \[R^t = \sum_{i=1}^n \nabla f_{\sigma_t\left(i\right)} \left(x^t_{i-1}\right) - \sum_{i=1}^n \nabla f_{\sigma_t\left(i\right)}\left(x^t_0\right).\]

Assume $F^*$ being the minimum of function $F(\cdot)$. For one epoch of \rsgd, we have
\begin{align}
F(x^{t+1}_0) - F^* &\leq F(x^t_0) - F^* - \gamma \left\langle \nabla F(x^t_0), n\nabla F(x^t_0) + R^t\right\rangle + \frac{L}{2}\gamma^2 \norm{n\nabla F(x^t_0) + R^t}^2 \nonumber\\
&\leq (1-2n\mu\gamma )\left[F(x^t_0) - F^* \right] - \gamma\left\langle \nabla F(x^t_0), R^t\right\rangle + \frac{L}{2}\gamma^2 \left[2n^2\norm{\nabla F(x^t_0)}^2 + 2\norm{R^t}^2\right]  \nonumber \\
&\leq (1-2n\mu\gamma + 2L^2n^2\gamma^2 )\left[F(x^t_0) - F^* \right]- \gamma\left\langle \nabla F(x^t_0), R^t\right\rangle + L\gamma^2 \norm{R^t}^2. \label{ll}
\end{align}
Here the second inequality is by the definition of Polyak-\L ojasiewicz condition, the last inequality uses the fact \[2L[F(x^t_0) - F^*] \geq \norm{\nabla F(x^t_0)}^2.\]
We have the following decomposition for the error term:
\begin{align}
R^t &= \sum_{i=1}^n \left[\nabla f_{\sigma_t\left(i\right)} \left(x^t_{i-1}\right) - \nabla f_{\sigma_t\left(i\right)}\left(x^t_{0}\right)\right]\nonumber\\
&= \sum_{i=1}^n \left[\int_{x^t_0}^{x^t_{i-1}} H_{\sigma_t\left(i\right)}\left(x\right) dx\right]\nonumber\\
&= \sum_{i=1}^n \left[\int_{x^t_0}^{x^t_{i-1}} H_{\sigma_t\left(i\right)}(x^t_0) dx\right] +  \sum_{i=1}^n \left[\int_{x^t_0}^{x^t_{i-1}} \left(H_{\sigma_t\left(i\right)}\left(x\right)-H_{\sigma_t\left(i\right)}(x^t_0)\right) dx\right]\nonumber\\
&=  \sum_{i=1}^n\left[H_{\sigma_t\left(i\right)}(x^t_0) \left(x^t_{i-1} - x^t_0\right)\right] +  \sum_{i=1}^n \left[\int_{x^t_0}^{x^t_{i-1}} \left(H_{\sigma_t\left(i\right)}\left(x\right)-H_{\sigma_t\left(i\right)}(x^t_0)\right) dx\right]\nonumber\\
&= \sum_{i=1}^n\left[H_{\sigma_t\left(i\right)}(x^t_0) \sum_{j=1}^{i-1}\left(-\gamma \nabla f_{\sigma_t\left(j\right)}\left(x^t_{j-1}\right)\right)\right] +  \sum_{i=1}^n \left[\int_{x^t_0}^{x^t_{i-1}} \left(H_{\sigma_t\left(i\right)}\left(x\right)-H_{\sigma_t\left(i\right)}(x^t_0)\right) dx\right]\nonumber\\
&= -\gamma \sum_{i=1}^n \left[H_{\sigma_t\left(i\right)}(x^t_0)\sum_{j=1}^{i-1} \nabla f_{\sigma_t\left(j\right)} \left(x^t_0\right)\right] - \gamma \sum_{i=1}^n \left\{H_{\sigma_t\left(i\right)}(x^t_0)\sum_{j=1}^{i-1}\left[\nabla f_{\sigma_t\left(j\right)} \left(x^t_{j-1}\right)-\nabla f_{\sigma_t\left(j\right)} \left(x^t_0\right)\right]\right\} \nonumber\\
&\ \ \ \ + \sum_{i=1}^n \left[\int_{x^t_0}^{x^t_{i-1}} \left(H_{\sigma_t\left(i\right)}\left(x\right)-H_{\sigma_t\left(i\right)}(x^t_0)\right) dx\right]\nonumber \\
&= A^t + B^t + C^t.\label{l1}
\end{align}
Here we define random variables \[A^t = -\gamma \sum_{i=1}^n \left[H_{\sigma_t\left(i\right)}(x^t_0)\sum_{j=1}^{i-1} \nabla f_{\sigma_t\left(j\right)} \left(x^t_0\right)\right],\]
\[B^t = - \gamma \sum_{i=1}^n \left\{H_{\sigma_t\left(i\right)}(x^t_0)\sum_{j=1}^{i-1}\left[\nabla f_{\sigma_t\left(j\right)} \left(x^t_{j-1}\right)-\nabla f_{\sigma_t\left(j\right)} \left(x^t_0\right)\right]\right\},\]
\[C^t = \sum_{i=1}^n \left[\int_{x^t_0}^{x^t_{i-1}} \left(H_{\sigma_t\left(i\right)}\left(x\right)-H_{\sigma_t\left(i\right)}(x^t_0)\right) dx\right].\]
There is 
\begin{align}
\Exp\left[A^t\right] = -\frac{n\left(n-1\right)}{2} \gamma \Exp\limits_{i\neq j}\left[ H_{i}(x^t_0) \nabla f_{j}\left(x^t_0\right)\right], \label{l2}
\end{align}
\begin{align}
\norm{B^t} &\leq \gamma \sum_{i=1}^n H_{\sigma_t\left(i\right)}(x^t_0) \sum_{j=1}^{i-1} \left(\nabla f_{\sigma_t\left(j\right)} \left(x^t_{j-1}\right) - \nabla f_{\sigma_t\left(j\right)} \left(x^t_0\right)\right)\nonumber\\
&\leq \gamma \sum_{i=1}^n L \sum_{j=1}^{i-1} \left(j-1\right)\gamma GL\nonumber\\
&= \gamma^2 L^2 G \sum_{i=1}^n \frac{\left(i-1\right)\left(i-2\right)}{2}\nonumber\\
&\leq \frac{1}{2} \gamma^2L^2Gn^3. \label{l3}
\end{align}
\begin{align}
\norm{C^t} &\leq \sum_{i=1}^n \left[\int_{0}^{\norm{x^t_{i-1}-x^t_0}} \norm{H_{\sigma_t\left(i\right)}\left(x^t_0+t\frac{x^t_{i-1}-x^t_0}{\norm{x^t_{i-1}-x^t_0}}\right)-H_{\sigma_t\left(i\right)}(x^t_0)}dt\right]\nonumber\\
&\leq \sum_{i=1}^n \left[L_H \norm{x^t_{i-1} - x^t_0}^2\right]\nonumber\\
&\leq n^3\gamma^2L_H G^2. \label{l4}
\end{align}
Using~\eqref{l1}~\eqref{l2}, we can decompose the innerproduct of $\nabla F(x^t_0 )$ and $\Exp\left[R^t\right]$ as following:
\begin{align}
-\gamma \left\langle \nabla F(x^t_0), \Exp\left[R^t\right] \right\rangle &= -\gamma \left\langle \nabla F(x^t_0), \Exp\left[A^t\right] + \Exp\left[B^t\right]+ \Exp\left[C^t\right] \right\rangle\nonumber\\
&= -\gamma \left\langle \nabla F(x^t_0), \Exp\left[A^t\right] \right\rangle - \gamma \left\langle \nabla F(x^t_0), \Exp\left[B^t\right] \right\rangle- \gamma \left\langle \nabla F(x^t_0), \Exp\left[C^t\right] \right\rangle\nonumber\\
&= \frac{1}{2}\gamma^2 n\left(n-1\right) \left\langle \nabla F(x^t_0), \Exp\limits_{i\neq j} H_{i}(x^t_0) \nabla f_{j}\left(x^t_0\right) \right\rangle -\gamma \left\langle \nabla F(x^t_0), \Exp\left[B^t\right] \right\rangle- \gamma \left\langle \nabla F(x^t_0), \Exp\left[C^t\right] \right\rangle.\label{l5}
\end{align}
For the first term in the~\eqref{l5}, we have further bound:
\begin{align}
&\ \ \ \ \frac{1}{2}\gamma^2 n\left(n-1\right) \left\langle \nabla F(x^t_0), \Exp\limits_{i\neq j} H_i(x^t_0) \nabla f_j\left(x^t_0\right)\right\rangle\nonumber\\
&= \frac{1}{2}\gamma^2 n^2 \left\langle \nabla F(x^t_0), \Exp\limits_{i,j} H_i(x^t_0) \nabla f_j\left(x^t_0\right)\right\rangle - \frac{1}{2}\gamma^2 n \left\langle \nabla F(x^t_0), \Exp\limits_{i} H_i(x^t_0) \nabla f_i\left(x^t_0\right)\right\rangle\nonumber\\
&\leq \frac{1}{2} \gamma^2n^2L\norm{\nabla F(x^t_0)}^2 + \frac{1}{8} \gamma n\frac{\mu}{L} \norm{\nabla F(x^t_0)}^2 + \frac{1}{2}\gamma^3 n \mu^{-1}L^3 G^2
\nonumber \\
&\leq ( \gamma^2n^2L^2 + \frac{1}{4} \gamma n\mu)[F(x^t_0) - F^*]+  \frac{1}{2}\gamma^3 n \mu^{-1}L^3 G^2.\label{l6}
\end{align}
For the second term in~\eqref{l5}, we use the bound
\begin{align}
-\gamma \left\langle \nabla F(x^t_0), \Exp\left[B^t\right]\right\rangle &\leq \frac{1}{8}\gamma\frac{\mu}{L}n\norm{\nabla F(x^t_0)}^2 + \frac{1}{2} \mu^{-1} \gamma^5 n^5 L^5 G^2\nonumber\\
&\leq \frac{1}{4}\gamma\mu n[F(x^t_0) - F^*] + \frac{1}{2} \mu^{-1} \gamma^5 n^5 L^5 G^2.\label{l7}
\end{align}
For the third term in~\eqref{l5}, we use the bound
\begin{align}
-\gamma \left\langle \nabla F(x^t_0), \Exp\left[C^t\right]\right\rangle &\leq \gamma\norm{\nabla F(x^t_0)}\cdot\left(  n^3\gamma^2L_H G^2 \right)\nonumber\\
&= \gamma^3 n^3 \norm{\nabla F(x^t_0)} L_H G^2\nonumber\\
&\leq \frac{1}{2}n^2\gamma^2 L \norm{\nabla F(x^t_0)}^2 + \frac{1}{2L} n^4\gamma^4L_H^2 G^4\nonumber\\
&\leq n^2\gamma^2 L^2 [F(x^t_0) - F^*] + \frac{1}{2L} n^4\gamma^4L_H^2 G^4.\label{l8}
\end{align}
Substituting~\eqref{l6}~\eqref{l7}~\eqref{l8} back to~\eqref{l5}, we get 
\begin{align}
-\gamma \left\langle \nabla F(x^t_0), \Exp\left[R^t\right] \right\rangle &\leq (\frac{1}{2}\gamma n \mu + 2n^2\gamma^2 L^2) [F(x^t_0) - F^*] + \frac{1}{2}\gamma^3 n \mu^{-1}L^3 G^2 +  \frac{1}{2} \mu^{-1} \gamma^5 n^5 L^5 G^2 +\frac{1}{2L} n^4\gamma^4L_H^2 G^4. \label{l9}
\end{align}
Substituting~\eqref{l9} to~\eqref{ll}, for one epoch we get recursion bound:
\begin{align}
&\ \ \ \ \Exp[F(x^t_0) - F^*] \nonumber\\
&\leq  (1-\frac{3}{2}n\mu\gamma + 4L^2n^2\gamma^2 )\left[F(x^t_0) - F^* \right] +  \frac{1}{2}\gamma^3 n \mu^{-1}L^3 G^2 +  \frac{1}{2} \mu^{-1} \gamma^5 n^5 L^5 G^2 +\frac{1}{2L} n^4\gamma^4L_H^2 G^4 + \frac{1}{4} n^4\gamma^4 G^2L^3. \label{l10}
\end{align}
Now assume \[ \frac{1}{2}n\mu \gamma > 4L^2n^2\gamma^2,\]  which we call assumption $1$, \eqref{l10} can be further turned into:
\begin{align}
&\ \ \ \ \Exp[F(x^t_0) - F^*] \nonumber\\
&\leq  (1-n\mu\gamma )\left[F(x^t_0) - F^* \right] +  \gamma^3 n C_1 +   n^4\gamma^4C_2 + n^5\gamma^5C_3. \label{l11}
\end{align}
where $C_1 =\frac{1}{2} \mu^{-1}L^3 G^2$, $C_2 =\frac{1}{2L} L_H^2 G^4 + \frac{1}{4}  G^2L^3$, $C_3 = \frac{1}{2} \mu^{-1} L^5 G^2$.
Further assume $n\gamma\mu<1$, which we call assumption $2$, expanding~\eqref{l11} over all the epochs we finally get a bound for \rsgd:
\begin{align*}
\Exp\norm{x_T - x^*}^2 \leq \left(1- n\gamma\mu\right)^{\frac{T}{n}}[F(x^t_0) - F^*] + \frac{T}{n} \left(\gamma^3nC_1 + \gamma^{4}n^{4}C_2 + \gamma^5n^5C_3\right).
\end{align*}
Let $\gamma = \frac{2\log T}{T\mu}$, there is 
\begin{align}
\Exp\norm{x_T - x^*}^2 &\leq \left(1-\frac{2n\log T}{T}\right)^{\frac{T}{2n\log T} 2\log T}[F(x_0) - F^*] + \frac{T}{n} \left(\gamma^3nC_1 + \gamma^{4}n^{4}C_2 + \gamma^5n^5C_3\right)\nonumber\\
&\leq \frac{1}{T^2}[F(x_0) - F^*]  + \frac{1}{T^2}\left(\log T\right)^3 C_4  + \frac{n^3}{T^3}\left(\log T\right)^4 C_5 + \frac{n^4}{T^4}\left(\log T\right)^5 C_6, \label{b12}
\end{align}
where $C_4 = \frac{8C_1}{\mu^3}$,  $C_5 = \frac{16C_2}{\mu^{4}}$, $C_6 = \frac{32C_2}{\mu^{5}}$.The second inequality comes from $\left(1-x\right)^{\frac{1}{x}} \leq \frac{1}{e}$ for $0<x<1$.
Obviously, this is a result of the form $\mathcal{O}\left(\frac{1}{T^2} + \frac{n^3}{T^3}\right)$. 

What remains to determine is to satisfy the two assumptions: (1) $\frac{1}{2}n\mu \gamma > 4L^2n^2\gamma^2$, (2) $n\gamma \mu <1$. 
The first is satisfied when\[\frac{T}{\log T} > 16 \frac{L^2}{\mu^2} n.\] The second assumption is satisfied when \[\frac{T}{\log T} > 2 n.\] Since $2<\frac{L}{\mu}$, the theorem is proved.
\end{proof}

\section{Proof of Theorem~\ref{ext1}}
\begin{proof}
For one epoch of \rsgd, We have the following inequality
\begin{align}
\norm{x^{t}_n - x^*}^2 &= \norm{x^t_0 - x^*}^2 - 2\gamma \left\langle x^t_0 - x^*, \sum_{i=1}^n \nabla f_{\sigma_t\left(i\right)} \left(x^t_{i-1}\right) \right\rangle + \gamma^2 \norm{\sum_{i=1}^n \nabla f_{\sigma_t\left(i\right)} \left(x^t_{i-1}\right)}^2\nonumber\\
&= \norm{x^t_0 - x^*}^2 - 2\gamma\left\langle x^t_0 - x^*, n\nabla F\left(x^t_0\right) \right\rangle - 2\gamma\left\langle x^t_0 - x^*, R^t\right\rangle + \gamma^2 \norm{n\nabla F\left(x^t_0\right) + R^t}^2 \nonumber\\
&\leq \norm{x^t_0-x^*}^2 - 2n\gamma\left[F(x^t_0) - F(x^*)\right] \nonumber -2\gamma\left\langle x^t_0 - x^*, R^t \right\rangle  + 2\gamma^2 n^2 \norm{\nabla F\left(x^t_0\right)}^2 + 2\gamma^2 \norm{R^t}^2\nonumber\\
&\leq \norm{x^t_0-x^*}^2 - \left(2n\gamma - 2n^2\gamma^2L\right)\left[F(x^t_0) - F(x^*)\right] -2\gamma\left\langle x^t_0 - x^*, R^t \right\rangle  + 2\gamma^2 \norm{R^t}^2, \label{e1}
\end{align}
where the first inequality is because of
\[\left\langle x^t_0 - x^*, \nabla F(x^t_0) \right\rangle \geq F(x^t_0) - F(x^*),\]
the second inequality is because of
\[\norm{\nabla F(x^t_0)}^2 \leq L \left[F(x)-F(x^*)\right].\]
Therefore, taking expectation of~\eqref{e1} leads to:
\begin{align}
\Exp[\norm{x^{t}_n - x^*}^2] \leq  \norm{x^t_0-x^*}^2 - (2n\gamma-2n^2\gamma^2L)\left[F(x^t_0) - F(x^*)\right] -2\gamma\Exp\left\langle x^t_0 - x^*, R^t \right\rangle  + 2\gamma^2\Exp\left[ \norm{R^t}^2\right], \label{e2}
\end{align}

Define random variables \[R^t_k = \sum_{i=1}^k \left[\nabla f_{\sigma_t\left(i\right)} \left(x^t_{i-1}\right) - \nabla f_{\sigma_t\left(i\right)}\left(x^t_{0}\right)\right],\]
where $1\leq k \leq n$. Obviously $R^t_n = R^t$. We firstly show that $\norm{R^t_k} \leq 3n^2L\gamma  (\norm{\nabla F(x^t_0)} + \delta)$, which is an important fact to be used in further analysis.

For any index $1\leq id \leq n$, there is \[\norm{\nabla f_{id}(x^t_1) - \nabla f_{id}(x^t_0)} \leq L\gamma(\norm{\nabla F(x^t_0)} + \delta).\]
Assume for any  $1\leq id \leq n$ and some $i$, there is (which is obviously true when $i=1$) \[\norm{\nabla f_{id}(x^t_i) - \nabla f_{id}(x^t_0)} \leq \left[\sum_{j=0}^{i-1} (1+L\gamma)^j\right] L\gamma(\norm{\nabla F(x^t_0)} + \delta).\]
Then for $i+1$, there is
\begin{nonumber}
\begin{align}
\norm{\nabla f_{id}(x^t_{i+1}) - \nabla f_{id}(x^t_0)}  &\leq \norm{\nabla f_{id}(x^t_{i}) - \nabla f_{id}(x^t_0)} + \norm{\nabla f_{id}(x^t_{i+1}) - \nabla f_{id}(x^t_i)}\\
&\leq \norm{\nabla f_{id}(x^t_{i}) - \nabla f_{id}(x^t_0)} + L\gamma (\norm{\nabla F(x^t_i)} + \delta)\\
&\leq \norm{\nabla f_{id}(x^t_{i}) - \nabla f_{id}(x^t_0)} + L\gamma (\norm{\nabla F(x^t_0)} + \delta) + L\gamma (\norm{\nabla F(x^t_i) - \nabla F(x^t_0)})\\
&\leq (1+L\gamma)\left[\sum_{j=0}^{i-1} (1+L\gamma)^j\right] L\gamma(\norm{\nabla F(x^t_0)} + \delta) + L\gamma (\norm{\nabla F(x^t_0)} + \delta) \\
&= \left[\sum_{j=0}^{i} (1+L\gamma)^j\right] L\gamma(\norm{\nabla F(x^t_0)} + \delta).
\end{align}
\end{nonumber}
So by induction, there is \[\norm{\nabla f_{id}(x^t_i) - \nabla f_{id}(x^t_0)} \leq \left[\sum_{j=0}^{i-1} (1+L\gamma)^j\right] L\gamma(\norm{\nabla F(x^t_0)} + \delta)\] for all $1\leq i \leq n$. Since $\gamma \leq \frac{1}{16nL} \leq \frac{1}{nL}$, there is $1+\gamma L \leq \frac{1}{n}$. Therefore, we have 
\begin{nonumber}
\begin{align}
\norm{\nabla f_{id}(x^t_i) - \nabla f_{id}(x^t_0)} &\leq \left[\sum_{j=0}^{i-1} (1+L\gamma)^j\right] L\gamma(\norm{\nabla F(x^t_0)} + \delta) \\
&\leq \left[n (1+\frac{1}{n})^n\right] L\gamma(\norm{\nabla F(x^t_0)} + \delta)\\
&\leq 3nL\gamma(\norm{\nabla F(x^t_0)} + \delta).
\end{align}
\end{nonumber}
Therefore, for any $1\leq k \leq n$, there is 
\begin{nonumber}
\begin{align}
\norm{R^t_k} &\leq \sum_{i=1}^k\norm{\nabla f_{\sigma_t\left(i\right)} \left(x^t_{i-1}\right) - \nabla f_{\sigma_t\left(i\right)}\left(x^t_{0}\right)}\\
&\leq \sum_{i=1}^k 3nL\gamma(\norm{\nabla F(x^t_0)} + \delta)\\
&\leq 3n^2L\gamma(\norm{\nabla F(x^t_0)} + \delta).
\end{align}
\end{nonumber}

Similar to the previous proof, we have the following decomposition for the error term:
\begin{align}
R^t &= \sum_{i=1}^n \left[\nabla f_{\sigma_t\left(i\right)} \left(x^t_{i-1}\right) - \nabla f_{\sigma_t\left(i\right)}\left(x^t_{0}\right)\right]\nonumber\\
&= \sum_{i=1}^n \left[\int_{x^t_0}^{x^t_{i-1}} H_{\sigma_t\left(i\right)}\left(x\right) dx\right]\nonumber\\
&= \sum_{i=1}^n \left[\int_{x^t_0}^{x^t_{i-1}} H_{\sigma_t\left(i\right)} dx\right] +  \sum_{i=1}^n \left[\int_{x^t_0}^{x^t_{i-1}} \left(H_{\sigma_t\left(i\right)}\left(x\right)-H_{\sigma_t\left(i\right)}\right) dx\right]\nonumber\\
&=  \sum_{i=1}^n\left[H_{\sigma_t\left(i\right)} \left(x^t_{i-1} - x^t_0\right)\right] +  \sum_{i=1}^n \left[\int_{x^t_0}^{x^t_{i-1}} \left(H_{\sigma_t\left(i\right)}\left(x\right)-H_{\sigma_t\left(i\right)}\right) dx\right]\nonumber\\
&= \sum_{i=1}^n\left[H_{\sigma_t\left(i\right)} \sum_{j=1}^{i-1}\left(-\gamma \nabla f_{\sigma_t\left(j\right)}\left(x^t_{j-1}\right)\right)\right] +  \sum_{i=1}^n \left[\int_{x^t_0}^{x^t_{i-1}} \left(H_{\sigma_t\left(i\right)}\left(x\right)-H_{\sigma_t\left(i\right)}\right) dx\right]\nonumber\\
&= -\gamma \sum_{i=1}^n \left[H_{\sigma_t\left(i\right)}\sum_{j=1}^{i-1} \nabla f_{\sigma_t\left(j\right)} \left(x^t_0\right)\right] - \gamma \sum_{i=1}^n \left\{H_{\sigma_t\left(i\right)}\sum_{j=1}^{i-1}\left[\nabla f_{\sigma_t\left(j\right)} \left(x^t_{j-1}\right)-\nabla f_{\sigma_t\left(j\right)} \left(x^t_0\right)\right]\right\} \nonumber\\
&\ \ \ \ + \sum_{i=1}^n \left[\int_{x^t_0}^{x^t_{i-1}} \left(H_{\sigma_t\left(i\right)}\left(x\right)-H_{\sigma_t\left(i\right)}\right) dx\right]\nonumber \\
&= A^t + B^t + C^t.\label{e3}
\end{align}
Here we define random variables \[A^t = -\gamma \sum_{i=1}^n \left[H_{\sigma_t\left(i\right)}\sum_{j=1}^{i-1} \nabla f_{\sigma_t\left(j\right)} \left(x^t_0\right)\right],\]
\[B^t = - \gamma \sum_{i=1}^n \left\{H_{\sigma_t\left(i\right)}\sum_{j=1}^{i-1}\left[\nabla f_{\sigma_t\left(j\right)} \left(x^t_{j-1}\right)-\nabla f_{\sigma_t\left(j\right)} \left(x^t_0\right)\right]\right\},\]
\[C^t = \sum_{i=1}^n \left[\int_{x^t_0}^{x^t_{i-1}} \left(H_{\sigma_t\left(i\right)}\left(x\right)-H_{\sigma_t\left(i\right)}\right) dx\right].\]
There is 
\begin{align}\Exp\left[A^t\right] = -\frac{n\left(n-1\right)}{2} \gamma \Exp\limits_{i\neq j}\left[ H_{i} \nabla f_{j}\left(x^t_0\right)\right],\label{e4}\end{align}
\begin{align}
\norm{B^t} &\leq \gamma \sum_{i=1}^n H_{\sigma_t\left(i\right)} \sum_{j=1}^{i-1} \norm{\nabla f_{\sigma_t\left(j\right)} \left(x^t_{j-1}\right) - \nabla f_{\sigma_t\left(j\right)} \left(x^t_0\right)}\nonumber\\
&\leq \gamma \sum_{i=1}^n L \sum_{j=1}^{i-1} 3nL\gamma(\norm{\nabla F(x^t_0)} + \delta)\nonumber\\
&\leq 3 \gamma^2L^2n^3(\norm{\nabla F(x^t_0)} + \delta).\label{e5}
\end{align}
\begin{align}
\norm{C^t} &\leq \sum_{i=1}^n \left[\int_{0}^{\norm{x^t_{i-1}-x^t_0}} \norm{H_{\sigma_t\left(i\right)}\left(x^t_0+t\frac{x^t_{i-1}-x^t_0}{\norm{x^t_{i-1}-x^t_0}}\right)-H_{\sigma_t\left(i\right)}}dt\right]\nonumber\\
&\leq \sum_{i=1}^n \left[L_H\max\left\{\norm{x^t_{i-1}-x^*}, \norm{x^t_{0}-x^*}\right\} \norm{x^t_{i-1} - x^t_0}\right]\nonumber\\
&\leq nL_H Dn\gamma G. \label{enew1}
\end{align}

Using~\eqref{e3} and ~\eqref{e4}, we can decompose the inner product of $x^t_0 - x^*$ and $\Exp\left[R^t\right]$ into:
\begin{align}
-2\gamma \left\langle x^t_0 - x^*, \Exp\left[R^t\right] \right\rangle &= -2\gamma \left\langle x^t_0 - x^*, \Exp\left[A^t\right] + \Exp\left[B^t\right]+ \Exp\left[C^t\right] \right\rangle\nonumber\\
&= -2\gamma \left\langle x^t_0 - x^*, \Exp\left[A^t\right] \right\rangle - 2\gamma \left\langle x^t_0 - x^*, \Exp\left[B^t\right] \right\rangle- 2\gamma \left\langle x^t_0 - x^*, \Exp\left[C^t\right] \right\rangle\nonumber\\
&= \gamma^2 n\left(n-1\right) \left\langle x^t_0 - x^*, \Exp\limits_{i\neq j} H_{i} \nabla f_{j}\left(x^t_0\right) \right\rangle -2\gamma \left\langle x^t_0 - x^*, \Exp\left[B^t\right] \right\rangle- 2\gamma \left\langle x^t_0 - x^*, \Exp\left[C^t\right] \right\rangle.\label{e6}
\end{align}

For the first term in~\eqref{e6}, there is
\begin{align}
&\ \ \ \ \gamma^2 n\left(n-1\right) \left\langle x^t_0 - x^*, \Exp\limits_{i\neq j} H_i \nabla f_j\left(x^t_0\right)\right\rangle\nonumber\\
&= \gamma^2 n\left(n-1\right)\Exp\limits_{i\neq j} \left\langle H_i \left(x^t_0 - x^*\right),  \nabla f_j\left(x^t_0\right) - \nabla f_j\left(x^*\right)\right\rangle + \gamma^2 n\left(n-1\right) \left\langle x^t_0 - x^*, \Exp\limits_{i\neq j} H_i \nabla f_j\left(x^*\right)\right\rangle\nonumber\\
&\leq \gamma^2n^2 \Exp\limits_{i, j} \left\langle \nabla f_i\left(x^t_0\right) - \nabla f_i\left(x^*\right),\nabla f_j\left(x^t_0\right) - \nabla f_j\left(x^*\right) \right\rangle + \gamma^2n\left(n-1\right)D\norm{\Delta} \nonumber\\
&\ \ \ \ \ \ \ \ +\gamma^2 n\left(n-1\right)\Exp\limits_{i\neq j} \left\langle H_i \left(x^t_0 - x^*\right) - \left(\nabla f_i\left(x^t_0\right) - \nabla f_i\left(x^*\right)\right),  \nabla f_j\left(x^t_0\right) - \nabla f_j\left(x^*\right)\right\rangle\nonumber\\
&\leq \gamma^2n^2 \norm{\nabla F\left(x^t_0\right)}^2 + \gamma^2n\left(n-1\right)D\norm{\Delta}  + \gamma^2n\left(n-1\right)L_H L\norm{x^t_0 - x^*}^3. \label{e7}
\end{align}
Here we introduce variable $\Delta = \Exp_{i\neq j} \left[H_i \nabla f_j(x^*)\right]$ for simplicity of notation, with $i, j$ uniformly sampled from all pairs of different indices. The last inequality is because of 
\begin{align*}
\norm{H_i\left(x^t_0 - x^*\right) - \left(\nabla f_i\left(x^t_0\right) - \nabla f_i\left(x^*\right)\right)} &=\norm{ H_i\left(x^t_0 - x^*\right) - \int_{x^*}^{x^t_0} H_i\left(x\right) dx}\\
&=\norm{\int_{x^*}^{x^t_0} \left(H_i - H_i\left(x\right)\right)dx}\\
&\leq\int_{0}^{\norm{x^t_0-x^*}} \norm{H_i - H_i\left(x^* + t\frac{x^t_0-x^*}{\norm{x^t_0-x^*}}\right)}dt\\
&\leq L_H \norm{x^t_0 - x^*}^2.
\end{align*}

For the second term in~\eqref{e6}, we use~\eqref{e5} and have the bound
\begin{align}
-2\gamma \left\langle x^t_0 - x^*, \Exp\left[B^t\right]\right\rangle &\leq 6\gamma^3n^3L^2D(\norm{\nabla F(x^t_0)} + \delta). \label{e8}
\end{align}

For the third term in~\eqref{e6}, we use~\eqref{enew1} and have the bound
\begin{align}
-2\gamma \left\langle x^t_0 - x^*, \Exp\left[C^t\right]\right\rangle &\leq 2\gamma^2n^2 L_HD^2G. \label{enew2}
\end{align}

Substituting~\eqref{e7}~\eqref{e8} and~\eqref{enew2} back to~\eqref{e6}, we get 
\begin{align}
-2\gamma \left\langle x^t_0 - x^*, \Exp\left[R^t\right] \right\rangle &\leq  \gamma^2n^2 \norm{\nabla F\left(x^t_0\right)}^2 + \gamma^2n\left(n-1\right)D\norm{\Delta} + 6\gamma^3n^3L^2D(\norm{\nabla F(x^t_0)} + \delta) \nonumber\\
&\ \ \ \ \ +\gamma^2n^2L_H (LD^3 + 2D^2G). \label{e10}
\end{align}

Furthermore, we have
\begin{nonumber}
\begin{align}
\Exp\left[\norm{R^t}^2\right] &\leq \left[3n^2L\gamma(\norm{\nabla F(x^t_0)} + \delta)\right]^2\\
&\leq 18n^4L^2\gamma^2 (\norm{\nabla F(x^t_0)}^2 + \delta^2).
\end{align}
\end{nonumber}

Inequality~\eqref{e2} can be simplified to:
\begin{align}
\Exp[\norm{x^{t}_n - x^*}^2] &\leq   \norm{x^t_0-x^*}^2 - (2n\gamma-3n^2\gamma^2L)\left[F(x^t_0) - F(x^*)\right] + \gamma^2n\left(n-1\right)D\norm{\Delta} +\gamma^2n^2L_H (LD^3 + 2D^2G)\nonumber\\
&\ \ \ \ + 6\gamma^3n^3L^2D(\norm{\nabla F(x^t_0)} + \delta)+ 36n^4L^2\gamma^4 (\norm{\nabla F(x^t_0)}^2 + \delta^2). \nonumber\\
&\leq  \norm{x^t_0-x^*}^2 - (2n\gamma-3n^2\gamma^2L)\left[F(x^t_0) - F(x^*)\right] + \gamma^2n\left(n-1\right)D\norm{\Delta}+\gamma^2n^2L_H (LD^3 + 2D^2G) \nonumber\\
&\ \ \ \ +12\gamma^2n^2 \norm{\nabla F(x^t_0)}^2 + 12\gamma^4n^4L^4D^2 + 6\gamma^3n^3L^2D\delta+ 36n^4L^2\gamma^4 (\norm{\nabla F(x^t_0)}^2 + \delta^2).\label{e11}
\end{align}
By the definition of $\gamma$, there is \[36n^4L^2\gamma^4 \leq n^2\gamma^2,\] \[16n^2\gamma^2L \leq n\gamma.\] So there is 
\begin{align}
\Exp[\norm{x^{t}_n - x^*}^2] &\leq \norm{x^t_0-x^*}^2 - (2n\gamma-16n^2\gamma^2L)\left[F(x^t_0) - F(x^*)\right] + \gamma^2n\left(n-1\right)D\norm{\Delta} \nonumber\\
&\ \ \ \  +\gamma^2n^2L_H (LD^3 + 2D^2G)+ 12\gamma^4n^4L^4D^2 + 6\gamma^3n^3L^2D\delta+ 36n^4L^2\gamma^4\delta^2.\nonumber\\
&\leq \norm{x^t_0-x^*}^2 - n\gamma\left[F(x^t_0) - F(x^*)\right] + \gamma^2n^2D\norm{\Delta} \nonumber\\
&\ \ \ \  +\gamma^2n^2L_H (LD^3 + 2D^2G)+ 12\gamma^4n^4L^4D^2 + 6\gamma^3n^3L^2D\delta + 36n^4L^2\gamma^4\delta^2.\nonumber
\end{align}

Furthermore, there is
\begin{align}
n\gamma \left[F(x^t_0) - F(x^*)\right] &\leq \norm{x^t_0-x^*}^2 - \Exp[\norm{x^{t}_n - x^*}^2]  + \gamma^2n^2\left(D\norm{\Delta} + L_H LD^3 + 2L_HD^2G\right)\nonumber \\
&\ \ \ \ + 6\gamma^3n^3L^2D\delta+ n^4\gamma^4 (12L^4D^2 + 36 L^2 \delta^2) . \label{e12}
\end{align}
Taking expectation of~\eqref{e12} and summing over all epochs, we have:
\begin{align}
T\gamma \left[F(\bar{x}) - F(x^*)\right] \leq D^2 + \gamma^2Tn(D\norm{\Delta}+L_H LD^3 + 2L_HD^2G) + T\gamma^3n^2L^26D\delta+ T\gamma^4n^3 (12L^4D^2 + 36 L^2 \delta^2). \label{e13}
\end{align}
Substituting the step size into~\eqref{e13}, we have
\begin{align}
F(\bar{x}) - F(x^*) &\leq \frac{D^2}{T}\max\left\{16nL, \sqrt{\frac{ Tn\left(\norm{\Delta}+L_H LD^2 + 2L_HDG\right)}{D}}, \left(\frac{Tn^2L^2\delta}{D}\right)^\frac{1}{3}, (Tn^3L^4)^\frac{1}{4}\right\}  \nonumber\\
&\ \ \ \ + \frac{D\sqrt{nD\left(\norm{\Delta} + L_H LD^2 + 2L_HDG\right)}}{\sqrt{T}}+ \frac{6D(D^2n^2L^2\delta)^\frac{1}{3}}{T^\frac{2}{3}} + \frac{n^\frac{3}{4}}{T^\frac{3}{4}} \left(12LD^2 + \frac{36\delta^2}{L}\right) \nonumber\\
&\leq \frac{2D\sqrt{nD\left(\norm{\Delta}+L_H LD^2 + 2L_HDG\right)}}{\sqrt{T}} + \frac{7D(D^2n^2L^2\delta)^\frac{1}{3}}{T^\frac{2}{3}} + \frac{n^\frac{3}{4}}{T^\frac{3}{4}} \left(13LD^2 + \frac{36\delta^2}{L}\right) + \frac{16D^2nL}{T}. \nonumber
\end{align}
Obviously, this result is of the form \[\frac{2D\sqrt{nD\left(\norm{\Delta}+L_H LD^2 + 2L_HDG\right)}}{\sqrt{T}} + \Oc\left(\left(\frac{n}{T}\right)^\frac{2}{3}\delta^\frac{1}{3} + \left(\frac{n}{T}\right)^\frac{3}{4}\right)\]

\end{proof}

\section{Proof of Theorem~\ref{last}}
\begin{proof}
For both \sgd and \rsgd, we use $s(i)$ to denote the index of component function picked in the $i$th iteration. We have the following inequality
\begin{nonumber}
\begin{align}
||x_{t} - x^*||^2 &= ||x_{t-1} - x^*||^2 -2\gamma \langle x_{t-1} - x^*, \nabla f_{s(t)}(x_{t-1}) \rangle + \gamma^2 ||\nabla f_{s(t)} (x_{t-1})||^2\\
&= ||x_{t-1} - x^*||^2 -2\gamma \langle x_{t-1} - x^*, \nabla f_{s(t)}(x_{t-1}) - \nabla f_{s(t)}(x^*) \rangle + \gamma^2 ||\nabla f_{s(t)} (x_{t-1})||^2\\
&\leq ||x_{t-1}-x^*||^2 -2\gamma (\frac{||\nabla f_{s(t)}(x_{t-1}) - \nabla f_{s(t)}(x^*)||^2}{L_{s(t)}+\mu_{s(t)}} + \frac{L_{s(t)} \mu_{s(t)}}{L_{s(t)}+\mu_{s(t)}} ||x_{t-1}-x^*||^2  ) + \gamma^2 ||\nabla f_{s(t)} (x_{t-1})||^2\\
&= (1-2\gamma\frac{L_{s(t)} \mu_{s(t)}}{L_{s(t)}+\mu_{s(t)}} )||x_{t-1}-x^*||^2 -\gamma (\frac{2}{L_{s(t)}+\mu_{s(t)}} -\gamma) ||\nabla f_{s(t)} (x_{t-1})||^2\\
&\leq (1-2\gamma\frac{L_{s(t)} \mu_{s(t)}}{L_{s(t)}+\mu_{s(t)}} + \mu_{s(t)}^2 \gamma^2 - 2\gamma \frac{\mu_{s(t)}^2}{L_{s(t)} + \mu_{s(t)}}) ||x_{t-1}-x^*||^2\\
&= (1-2\gamma\mu_{s(t)}+\mu_{s(t)}^2\gamma^2) ||x_{t-1}-x^*||^2\\
&= (1-\gamma\mu_{s(t)})^2 ||x_{t-1}-x^*||^2.
\end{align}
\end{nonumber}
So we have \[\Exp||x_T - x^*||^2 \leq \Exp[\prod_{i=1}^T (1-\gamma\mu_{s(t)})^2] ||x_{t-1}-x^*||^2.\] By the AM-GM inequality, we know the term $\Exp[\prod_{i=1}^T (1-\gamma\mu_{s(t)})^2] $ for \rsgd is no larger than that of \sgd. Also, this bound is tight when we consider $f_i(x) = \frac{\mu_i}{2}||x-x^*||^2$, which completes the proof. 

\end{proof}

\section{\sgd under Polyak-\L ojasiewicz condition}

For the completeness of the paper, we include the following analysis of \sgd under Polyak-\L ojasiewicz condition.
\begin{theorem}
For finite sum problem satisfying Polyak-\L ojasiewicz condition with parameter $\mu$, Lipschitz constant $L$, setting step size \[\gamma = \frac{\log T}{\mu T}, \]
there is \[F(x_T) - F^* \leq \Oc(\frac{1}{T}).\]
\end{theorem}
\begin{proof}
We have the following one iteration for SGD with step size $\gamma$: 
\begin{align}
x_{t+1} = x_t - \gamma \nabla f_{s(t)}(x_t).
\end{align}

Given $x_t$, there is randomness over index 
\begin{align}
\Exp [F(x_{t+1})] - F^* &\leq F(x_t) - \gamma \Exp [\langle \nabla F(x_t), \nabla f_{s(t)}(x_t) \rangle] + \frac{L}{2} \gamma^2 \Exp[ \norm{\nabla f_i(x_t)}^2] - F^*\\
&= F(x_t) - \gamma  \langle \nabla F(x_t), \nabla F(x_t) \rangle +\frac{L}{2}  \gamma^2 \Exp[ \norm{\nabla f_i(x_t)}^2]- F^*\\
&\leq F(x_t) - \gamma \mu [F(x_t) - F^*] + \frac{L}{2} \gamma^2 G^2- F^*\\
&= (1-2\gamma \mu) [F(x_t) - F^*] + \frac{L}{2} \gamma^2 G^2.
\end{align}

The first inequality is because \[F(x) \leq F(y) + \langle x-y, \nabla F(y)  \rangle + \frac{L}{2} \norm{x-y}^2.\]
The second inequality is because of the definition of LP condition.

Expanding over iterations leads to

\[\Exp[F(x_T) - F^*] \leq (1-2\gamma \mu)^T [F(x_0) - F^*] + \frac{L}{2} T \gamma^2 G^2.\]

Setting $\gamma = \frac{\log T}{\mu T}$ leads to a $O(\frac{1}{T})$ convergence of $F(x_T) - F^*$. 
\end{proof}

%\section{Proof of lemma~\ref{lm1}}
%For the consistence, we include the proof proof for the following standard lemma.
%\begin{lemma}\label{lm1}
%Assume $f: \mathbb{R}^d\rightarrow \mathbb{R}$ is strongly $\mu-$convex function and has $L-$Lipschitz continuous gradient, for $x,y \in \mathbb{R}^d$, there is $$\langle\nabla f(x) - \nabla f(y), x-y\rangle\geq \frac{1}{L+\mu}\norm{\nabla f(x) - \nabla f(y)}^2 + \frac{L\mu}{L+\mu}\norm{x-y}^2.$$
%\end{lemma}
%\begin{proof}
%When $\mu = L$, the result is obvious by AM-GM inequality.

%Now we prove the lemma for $\mu < L$. Since $f$ is $\mu$ strongly convex, there is $f(x) - \frac{1}{2}\mu\norm{x}^2$ being a convex function with $L-\mu$ Lipschitz gradient. Therefore, we have

%\[\left\langle\nabla f(x)-\nabla f(y) - \mu x + \mu y, x-y \right\rangle\geq \frac{1}{L-\mu} \norm{\nabla f(x) - \nabla f(y) - \mu x +\mu y}^2.\]
%So there is
%\[\left\langle\nabla f(x)-\nabla f(y) , x-y \right\rangle - \mu \norm{x-y}^2 \geq \frac{1}{L-\mu}\norm{\nabla f(x) - \nabla f(y)}^2 - \frac{2\mu}{L-\mu} \left\langle\nabla f(x)-\nabla f(y), x-y\right\rangle +\frac{\mu^2}{L-\mu}\norm{x-y}^2.\]
%Therefore, \[\frac{L+\mu}{L-\mu}\left\langle\nabla f(x)-\nabla f(y) , x-y \right\rangle \geq \frac{1}{L-\mu}\norm{\nabla f(x) - \nabla f(y)}^2 +\frac{L\mu}{L-\mu}\norm{x-y}^2.\]
%As a result, we have
%\[\left\langle\nabla f(x) - \nabla f(y), x-y\right\rangle\geq \frac{1}{L+\mu}\norm{\nabla f(x) - \nabla f(y)}^2 + \frac{L\mu}{L+\mu}\norm{x-y}^2.\]

%\end{proof}

%% file: ms.bbl
\begin{thebibliography}{38}
\providecommand{\natexlab}[1]{#1}
\providecommand{\url}[1]{\texttt{#1}}
\expandafter\ifx\csname urlstyle\endcsname\relax
  \providecommand{\doi}[1]{doi: #1}\else
  \providecommand{\doi}{doi: \begingroup \urlstyle{rm}\Url}\fi

\bibitem[Arjevani and Shamir(2016)]{arjevani2016dimension}
Y.~Arjevani and O.~Shamir.
\newblock Dimension-free iteration complexity of finite sum optimization
  problems.
\newblock In \emph{Advances in Neural Information Processing Systems}, pages
  3540--3548, 2016.

\bibitem[Bertsekas(2011)]{bertsekas2011incremental}
D.~P. Bertsekas.
\newblock Incremental gradient, subgradient, and proximal methods for convex
  optimization: A survey.
\newblock \emph{Optimization for Machine Learning}, 2010\penalty0
  (1-38):\penalty0 3, 2011.

\bibitem[Bottou(2009)]{bottou2009curiously}
L.~Bottou.
\newblock Curiously fast convergence of some stochastic gradient descent
  algorithms.
\newblock In \emph{Proceedings of the symposium on learning and data science,
  Paris}, 2009.

\bibitem[Bottou(2012)]{bottou2012stochastic}
L.~Bottou.
\newblock Stochastic gradient descent tricks.
\newblock In \emph{Neural networks: Tricks of the trade}, pages 421--436.
  Springer, 2012.

\bibitem[Bottou et~al.(2016)Bottou, Curtis, and
  Nocedal]{bottou2016optimization}
L.~Bottou, F.~E. Curtis, and J.~Nocedal.
\newblock Optimization methods for large-scale machine learning.
\newblock \emph{arXiv:1606.04838}, 2016.

\bibitem[De and Goldstein(2016)]{de2016efficient}
S.~De and T.~Goldstein.
\newblock Efficient distributed sgd with variance reduction.
\newblock In \emph{Data Mining (ICDM), 2016 IEEE 16th International Conference
  on}, pages 111--120. IEEE, 2016.

\bibitem[Defazio et~al.(2014{\natexlab{a}})Defazio, Bach, and
  Lacoste-Julien]{defazio2014saga}
A.~Defazio, F.~Bach, and S.~Lacoste-Julien.
\newblock Saga: A fast incremental gradient method with support for
  non-strongly convex composite objectives.
\newblock In \emph{Advances in neural information processing systems}, pages
  1646--1654, 2014{\natexlab{a}}.

\bibitem[Defazio et~al.(2014{\natexlab{b}})Defazio, Domke,
  et~al.]{defazio2014finito}
A.~Defazio, J.~Domke, et~al.
\newblock Finito: A faster, permutable incremental gradient method for big data
  problems.
\newblock In \emph{International Conference on Machine Learning}, pages
  1125--1133, 2014{\natexlab{b}}.

\bibitem[Feng et~al.(2012)Feng, Kumar, Recht, and R{\'e}]{feng2012towards}
X.~Feng, A.~Kumar, B.~Recht, and C.~R{\'e}.
\newblock Towards a unified architecture for in-rdbms analytics.
\newblock In \emph{Proceedings of the 2012 ACM SIGMOD International Conference
  on Management of Data}, pages 325--336. ACM, 2012.

\bibitem[G{\"u}rb{\"u}zbalaban et~al.(2015{\natexlab{a}})G{\"u}rb{\"u}zbalaban,
  Ozdaglar, and Parrilo]{gurbuzbalaban2015convergence}
M.~G{\"u}rb{\"u}zbalaban, A.~Ozdaglar, and P.~Parrilo.
\newblock Convergence rate of incremental gradient and newton methods.
\newblock \emph{arXiv preprint arXiv:1510.08562}, 2015{\natexlab{a}}.

\bibitem[G{\"u}rb{\"u}zbalaban et~al.(2015{\natexlab{b}})G{\"u}rb{\"u}zbalaban,
  Ozdaglar, and Parrilo]{gurbuzbalaban2015random}
M.~G{\"u}rb{\"u}zbalaban, A.~Ozdaglar, and P.~Parrilo.
\newblock Why random reshuffling beats stochastic gradient descent.
\newblock \emph{arXiv preprint arXiv:1510.08560}, 2015{\natexlab{b}}.

\bibitem[G{\"u}rb{\"u}zbalaban et~al.(2017)G{\"u}rb{\"u}zbalaban, Ozdaglar,
  Parrilo, and Vanli]{Grbzbalaban2017WhenCC}
M.~G{\"u}rb{\"u}zbalaban, A.~E. Ozdaglar, P.~A. Parrilo, and N.~D. Vanli.
\newblock When cyclic coordinate descent outperforms randomized coordinate
  descent.
\newblock In \emph{NIPS}, 2017.

\bibitem[Hazan and Kale(2014)]{hazan2014beyond}
E.~Hazan and S.~Kale.
\newblock Beyond the regret minimization barrier: optimal algorithms for
  stochastic strongly-convex optimization.
\newblock \emph{The Journal of Machine Learning Research}, 15\penalty0
  (1):\penalty0 2489--2512, 2014.

\bibitem[Israel et~al.(2016)Israel, Krahmer, and Ward]{israel2016arithmetic}
A.~Israel, F.~Krahmer, and R.~Ward.
\newblock An arithmetic--geometric mean inequality for products of three
  matrices.
\newblock \emph{Linear Algebra and its Applications}, 488:\penalty0 1--12,
  2016.

\bibitem[Johnson and Zhang(2013)]{johnson2013accelerating}
R.~Johnson and T.~Zhang.
\newblock Accelerating stochastic gradient descent using predictive variance
  reduction.
\newblock In \emph{Advances in neural information processing systems}, pages
  315--323, 2013.

\bibitem[Kohonen(1974)]{kohonen1974adaptive}
T.~Kohonen.
\newblock An adaptive associative memory principle.
\newblock \emph{IEEE Transactions on Computers}, 100\penalty0 (4):\penalty0
  444--445, 1974.

\bibitem[Lee and Wright(2016)]{lee2016random}
C.-P. Lee and S.~J. Wright.
\newblock Random permutations fix a worst case for cyclic coordinate descent.
\newblock \emph{arXiv preprint arXiv:1607.08320}, 2016.

\bibitem[Lee et~al.(2015)Lee, Lin, Ma, and Yang]{lee2015distributed}
J.~D. Lee, Q.~Lin, T.~Ma, and T.~Yang.
\newblock Distributed stochastic variance reduced gradient methods and a lower
  bound for communication complexity.
\newblock \emph{arXiv preprint arXiv:1507.07595}, 2015.

\bibitem[Moulines and Bach(2011)]{moulines2011non}
E.~Moulines and F.~R. Bach.
\newblock Non-asymptotic analysis of stochastic approximation algorithms for
  machine learning.
\newblock In \emph{Advances in Neural Information Processing Systems}, pages
  451--459, 2011.

\bibitem[Nedi{\'c} and Bertsekas(2001)]{nedic2001convergence}
A.~Nedi{\'c} and D.~Bertsekas.
\newblock Convergence rate of incremental subgradient algorithms.
\newblock In \emph{Stochastic optimization: algorithms and applications}, pages
  223--264. Springer, 2001.

\bibitem[Nemirovski et~al.(2009)Nemirovski, Juditsky, Lan, and
  Shapiro]{nemirov09}
A.~Nemirovski, A.~Juditsky, G.~Lan, and A.~Shapiro.
\newblock Robust stochastic approximation approach to stochastic programming.
\newblock \emph{SIAM Journal on Optimization}, 19\penalty0 (4):\penalty0
  1574--1609, 2009.

\bibitem[Nemirovskii et~al.(1983)Nemirovskii, Yudin, and
  Dawson]{nemirovskii1983problem}
A.~Nemirovskii, D.~B. Yudin, and E.~R. Dawson.
\newblock \emph{Problem complexity and method efficiency in optimization}.
\newblock Wiley, 1983.

\bibitem[Nesterov(2013)]{nesterov2013introductory}
Y.~Nesterov.
\newblock \emph{Introductory lectures on convex optimization: A basic course},
  volume~87.
\newblock Springer Science \& Business Media, 2013.

\bibitem[Nesterov and Polyak(2006)]{nesterov2006cubic}
Y.~Nesterov and B.~T. Polyak.
\newblock Cubic regularization of newton method and its global performance.
\newblock \emph{Mathematical Programming}, 108\penalty0 (1):\penalty0 177--205,
  2006.

\bibitem[Polyak(1963)]{polyak1963gradient}
B.~T. Polyak.
\newblock Gradient methods for the minimisation of functionals.
\newblock \emph{USSR Computational Mathematics and Mathematical Physics},
  3\penalty0 (4):\penalty0 864--878, 1963.

\bibitem[Rakhlin et~al.(2012)Rakhlin, Shamir, Sridharan,
  et~al.]{rakhlin2012making}
A.~Rakhlin, O.~Shamir, K.~Sridharan, et~al.
\newblock Making gradient descent optimal for strongly convex stochastic
  optimization.
\newblock In \emph{ICML}. Citeseer, 2012.

\bibitem[Recht and R\'e(2012)]{Recht2012BeneathTV}
B.~Recht and C.~R\'e.
\newblock Beneath the valley of the noncommutative arithmetic-geometric mean
  inequality: conjectures, case-studies, and consequences.
\newblock \emph{arXiv preprint arXiv:1202.4184}, 2012.

\bibitem[Recht et~al.(2011)Recht, Re, Wright, and Niu]{recht2011hogwild}
B.~Recht, C.~Re, S.~Wright, and F.~Niu.
\newblock Hogwild: A lock-free approach to parallelizing stochastic gradient
  descent.
\newblock In \emph{Advances in neural information processing systems}, pages
  693--701, 2011.

\bibitem[Reddi et~al.(2016)Reddi, Hefny, Sra, Poczos, and
  Smola]{reddi2016stochastic}
S.~J. Reddi, A.~Hefny, S.~Sra, B.~Poczos, and A.~Smola.
\newblock Stochastic variance reduction for nonconvex optimization.
\newblock In \emph{International conference on machine learning}, pages
  314--323, 2016.

\bibitem[Shalev-Shwartz and Ben-David(2014)]{shalev2014understanding}
S.~Shalev-Shwartz and S.~Ben-David.
\newblock \emph{Understanding machine learning: From theory to algorithms}.
\newblock Cambridge university press, 2014.

\bibitem[Shalev-Shwartz and Zhang(2013)]{shalev2013stochastic}
S.~Shalev-Shwartz and T.~Zhang.
\newblock Stochastic dual coordinate ascent methods for regularized loss
  minimization.
\newblock \emph{Journal of Machine Learning Research}, 14\penalty0
  (Feb):\penalty0 567--599, 2013.

\bibitem[Shamir(2016)]{shamir2016without}
O.~Shamir.
\newblock Without-replacement sampling for stochastic gradient methods.
\newblock In \emph{Advances in Neural Information Processing Systems}, pages
  46--54, 2016.

\bibitem[Solodov(1998)]{solodov1998incremental}
M.~V. Solodov.
\newblock Incremental gradient algorithms with stepsizes bounded away from
  zero.
\newblock \emph{Computational Optimization and Applications}, 11\penalty0
  (1):\penalty0 23--35, 1998.

\bibitem[Sra et~al.(2012)Sra, Nowozin, and Wright]{sra2012optimization}
S.~Sra, S.~Nowozin, and S.~J. Wright.
\newblock \emph{Optimization for machine learning}.
\newblock Mit Press, 2012.

\bibitem[Sun and Ye(2016)]{sun2016worst}
R.~Sun and Y.~Ye.
\newblock Worst-case complexity of cyclic coordinate descent: $o (n^ 2)$ gap
  with randomized version.
\newblock \emph{arXiv preprint arXiv:1604.07130}, 2016.

\bibitem[Wright and Lee(2017)]{wright2017analyzing}
S.~J. Wright and C.-P. Lee.
\newblock Analyzing random permutations for cyclic coordinate descent.
\newblock \emph{arXiv preprint arXiv:1706.00908}, 2017.

\bibitem[Ying et~al.(2018)Ying, Yuan, Vlaski, and Sayed]{ying2018stochastic}
B.~Ying, K.~Yuan, S.~Vlaski, and A.~H. Sayed.
\newblock Stochastic learning under random reshuffling.
\newblock \emph{arXiv preprint arXiv:1803.07964}, 2018.

\bibitem[Zhang(2014)]{zhang2014note}
T.~Zhang.
\newblock A note on the non-commutative arithmetic-geometric mean inequality.
\newblock \emph{arXiv:1411.5058}, 2014.

\end{thebibliography}
